\documentclass[11pt]{amsart}
\usepackage{amsmath,amscd,amssymb}
\usepackage{hyperref}
\newtheorem{theorem}{Theorem}[section]

\newtheorem{lemma}[theorem]{Lemma}
\newtheorem{proposition}[theorem]{Proposition}
\newtheorem{remark}[theorem]{Remark}
\numberwithin{equation}{section}
\newtheorem{definition}[theorem]{Definition}
\newtheorem{example}[theorem]{Example}
\newtheorem{corollary}[theorem]{Corollary}

\input xy
\xyoption{all}

\begin{document}
\title[Initialized and ACM line bundles on  sextic surfaces]
{On initialized and ACM line bundles over a smooth sextic surface in $\mathbb P^3$}

\author[D. Bhattacharya]{Debojyoti Bhattacharya }
\address{Indian Institute of Science Education and Research,Thiruvananthapuram,
Maruthamala PO, Vithura,
Thiruvananthapuram - 695551, Kerala, India}
\email{debojyotibhattacharya15@iisertvm.ac.in}

\date{05.07.2021}

\begin{abstract}
Let $X \subset \mathbb P^3$ be a smooth hypersurface of degree $6$ over  complex numbers.  In this paper, we give a characterization of  initialized and ACM  bundles of rank $1$ on $X$ with respect to the line bundle given by a smooth hyperplane section of $X$.
\end{abstract}

\keywords{ACM line bundle, sextic hypersurface, curve, Clifford-index, gonality, $1$-connectedness}

\subjclass[2010]{14J70.14J60}



\maketitle\section{Introduction and the statement of the main result}

Let $X \subset \mathbb P^n$ be a smooth projective variety and $H$ be the very ample line bundle given by a smooth hyperplane section $C$ of $X$. If the coordinate ring of $X$ is Cohen-Macaulay, then a vector bundle $\mathcal F$ on $X$ is called \textit{arithmetically Cohen-Macaulay}(ACM) if it has no intermediate cohomology, i.e. if :\

$H^i(\mathcal F \otimes H^{\otimes t})=0$, for $1 \leq i\leq \text{dim}(X)-1$ and for all $t \in \mathbb Z$.\

It is always interesting to ask whether one can classify ACM bundles on a fixed variety. On a projective space $\mathbb P^n$, a higher rank vector bundle is ACM iff it's obtained as a direct sum of line bundles on $\mathbb P^n$(see \cite{Horrocks}). This is not necessarily true for more general polarized smooth projective varieties (e.g. hypersurfaces of degree $d\geq 2$ in $\mathbb P^n$). In this direction, considerable efforts have  been driven towards a classification of ACM bundles on hypersurfaces $X^{(d)} \subset \mathbb P^n$  of degree $d$. For $d=2$, ACM bundles are classified by Kn\"{o}rrer (see \cite{quadric}). For $d=3$, Casanellas and Hartshorne (see \cite{H3}) have constructed a $n^2 + 1$-dimensional family of rank $n$ indecomposable ACM vector bundles on $X^{(d)}$ with Chern classes $c_1 = nH$ and  $c_2 = \frac{1}{2}(3n^2-n)$ for $n \geq 2$ and Faenzi (see \cite{F3}) gave a complete classification of rank $2$ ACM bundles on $X^{(d)}$. For $d=4, n=3$, Coskun, Kulkarni and Mustopa have constructed a $14$-dimensional family of simple Ulrich bundles on $X^{(d)}$ of rank $2$ with $c_1 = H^{\otimes 3}$ and $c_2 = 14$, in the case where $X^{(d)}$ is a Pfaffian quartic surface (see \cite{CKM}). For $d=4,n=3$, Casnati has classified indecomposable  ACM bundles of rank $2$ on $X^{(d)}$, where $X^{(d)}$ is general determinantal (see \cite{Casnati}). For $d=5,n=3$, rank $2$ ACM  bundles on general such $X^{(d)}$ are classified by Chiantini and Faenzi (see \cite{CF5}). For $d=6,n=3$, rank $2$ ACM  bundles on general such $X^{(d)}$  are classified by M. Patnott (see \cite{MP6}). 


 In this context a characterization of ACM  bundles of rank $1$ on $X$ is useful for the construction of higher rank indecomposable ACM bundles on $X$. In this direction Pons-Llopis and Tonini have classified ACM line bundles on a DelPezzo
surface $X$ with respect to the anti-canonical line bundle on $X$ (see \cite{delpezzo}). F.Chindea has classified ACM line bundles on a complex polarized elliptic ruled surface (see \cite{Chindea}). Recently K.Watanabe has classified ACM line bundles on smooth quartic hypersurfaces in $\mathbb P^3$ (see \cite{W4}), on polarized $K3$ surfaces (see \cite{WP4}) and on smooth quintic hypersurfaces in $\mathbb P^3$ (see \cite{W5}). This motivates us to extend the study related to the classification of initialized (i.e., line bundles $\mathcal L$ with $H^0(X,\mathcal L) \neq 0, H^0(X,\mathcal L \otimes H^*) = 0$) and ACM line bundles on  smooth sextic hypersurfaces in $\mathbb P^3$. Our main result is as follows :\\

\begin{theorem}\label{T1}
Let $X \subset \mathbb P^3$ be a smooth sextic hypersurface. Let $H$ be the hyperplane class of $X$ and $C \in |H|$ be a smooth member. Let $D$ be a non-zero effective divisor on $X$ of arithmetic genus $P_a(D)$. Then $\mathcal O_X(D)$ is initialized and ACM if and only if the following conditions are satisfied :\\

\hbox{$(i)$ either $2C.D-9 \leq P_a(D) \leq 2C.D-4 $ or $P_a(D)= 2C.D-2$.}.







\hbox{$(ii)$ If $P_a(D)=2C.D-2$, then $C.D=1$.}.

\hbox{$(iii)$ If $P_a(D) =2C.D-4$, then the following occurs:}

\hbox{$(a)$ $C.D=2$ or $5 \leq C.D \leq 6$ or $12 \leq C.D \leq 15$.}

\hbox{$(b)$ If $C.D \in \{6,12,13,14,15\}$, then $h^0(\mathcal O_C(D-C)) =0$ and if $12 \leq C.D \leq 15$, then  $h^0(\mathcal O_C(D-2C)) =0$.}

\hbox{$(c)$ If $C.D=12$, then $h^0(\mathcal O_C(2C-D))=0$ and if   $C.D=15$, then $h^0(\mathcal O_C(D))=6$.}.

\hbox{$(iv)$ If $P_a(D) =2C.D-5$, then the following occurs :}

\hbox{$(a)$ $3 \leq C.D \leq 7$ or $10 \leq C.D \leq 14$.}

\hbox{$(b)$ If $C.D \in \{6,7, 10,11,12,13,14\}$, then $h^0(\mathcal O_C(D-C)) =0$ and for $C.D \in \{11,12\}$, $  h^0(\mathcal O_C(2C-D))=0$.}

\hbox{$(c)$ If $C.D=6$, then $(h^0(\mathcal O_C(D)),h^0(\mathcal O_C(2C-D)) \neq (3,3)$ and $h^0(\mathcal O_C(C-D))=0$.}

\hbox{$(d)$ If $12 \leq C.D \leq 14$, then $h^0(\mathcal O_C(D))=5$.}.

\hbox{$(v)$ If $P_a(D)=2C.D-6$, then the following occurs :}

\hbox{$(a)$ $3 \leq C.D \leq 13$.}

\hbox{$(b)$ If $6 \leq C.D \leq 13$, then $h^0(\mathcal O_C(D-C))=0$ and if $10 \leq C.D \leq 12$, then $h^0(\mathcal O_C(2C-D)) =0$.}

\hbox{$(c)$ If $3 \leq C.D \leq 4$, then $(h^0(\mathcal O_C(C-D)), h^0(\mathcal O_C(2C-D))) \neq (1,3)$ and $h^0(\mathcal O_C(C-D))=0$ for $C.D=6$.}


\hbox{$(d)$ If $C.D =7$, then  $(h^0(\mathcal O_C(D)), h^0(\mathcal O_C(2C-D))) \neq (3,2)$.}

\hbox{$(e)$ If $C.D=6$, then  $(h^0(\mathcal O_C(D)), h^0(\mathcal O_C(2C-D))) \neq (i,j)$, where $i,j \in \{2,3\}, i+j \in \{5,6\}$.}

\hbox{$(f)$ If $C.D =5$, then  $(h^0(\mathcal O_C(D)), h^0(\mathcal O_C(2C-D))) \neq (2,3)$ and  if  $11 \leq C.D \leq 13$, then $h^0(\mathcal O_C(D))=4$.}.




\hbox{$(vi)$ If $P_a(D)=2C.D-7$, then the following occurs:}

\hbox{$(a)$ $4 \leq C.D \leq 12$.}

\hbox{$(b)$ If $6 \leq C.D \leq 12$, then $h^0(\mathcal O_C(D-C)) =0$ and if $10 \leq C.D \leq 12$, then $h^0(\mathcal O_C(2C-D))=0$.}

\hbox{$(c)$ If $8 \leq C.D \leq 9$, then $(h^0(\mathcal O_C(D)), h^0(\mathcal O_C(2C-D))) \neq (3,1)$.}


\hbox{$(d)$ If $C.D=7$, then $(h^0(\mathcal O_C(D)), h^0(\mathcal O_C(2C-D))) \neq (i,j)$, where $i \in \{2,3\}, j\in \{1,2\}$, $i+j \in \{4,5\}$.}

\hbox{$(e)$ If $C.D=6$, then  $(h^0(\mathcal O_C(D)), h^0(\mathcal O_C(2C-D))) \neq (i,j)$, where $i,j \in \{1,2,3\}, i+j \in \{4,5,6\}$.}

\hbox{$(f)$ If $C.D=5$, then $(h^0(\mathcal O_C(D)), h^0(\mathcal O_C(2C-D))) \neq (i,j)$, where $i \in \{1,2\}, j\in \{2,3\}, i+j \in\{4,5\}$.}

\hbox{$(g)$ If $C.D=4$, then $h^0(\mathcal O_C(2C-D)) \neq 3$.}

\hbox{$(h)$ If $C.D \in \{4,6\}$, then $h^0(\mathcal O_C(C-D)) =0$ and if $10 \leq C.D \leq 12$, then $h^0(\mathcal O_C(D))=3$.}.


\hbox{$(vii)$ If $P_a(D)=2C.D-8$, then the following occurs :}

\hbox{$(a)$ $4 \leq C.D \leq 11$.}

\hbox{$(b)$ If $5 \leq C.D \leq 9$, then  $(h^0(\mathcal O_C(D)), h^0(\mathcal O_C(2C-D))) \neq (2,1)$ and if $6 \leq C.D \leq 9$, then $h^0(\mathcal O_C(D)) \neq 3$.}

\hbox{$(c)$ If $4 \leq C.D \leq 6$, then  $h^0(\mathcal O_C(C-D))=0$ and if $5 \leq C.D \leq 7$, then $h^0(\mathcal O_C(2C-D)) \notin \{2,3\}$.}

\hbox{$(d)$ If $C.D=4$, then $h^0(\mathcal O_C(2C-D)) =1$ and if $10 \leq C.D \leq 11$, then $(h^0(\mathcal O_C(D)), h^0(\mathcal O_C(2C-D))=(2,0)  $.}.



\hbox{$(viii)$ If $P_a(D)=2C.D-9$, then the following occurs :}

\hbox{$(a)$ $5 \leq C.D \leq 10$, $h^0(\mathcal O_C(D))=1$ and $h^0(\mathcal O_C(2C-D)) =0$.}

\hbox{$(b)$ If $5 \leq C.D \leq 6$, then $h^0(\mathcal O_C(C-D))=0$.}.



\end{theorem}

\subsection{The plan of the paper}

Our plan of the paper is as follows :

Section \ref{S2} is dedicated to showcase the notions and results that will be useful in the next section. This section can be roughly seen as divided into three parts. In the first part, we begin by recalling definitions and results concerning some important invariants of line bundles and smooth plane curves. Here we also document results related to the existence  of certain divisors on smooth plane curves. The second part  deals with the results on line bundles over a smooth hypersurface in $\mathbb P^3$. To be precise, after mentioning basic facts on line bundles on smooth hypersurfaces, we note down  a sufficient condition for the non-negativity of arithmetic genus of an effective divisor on surfaces of degree $d \geq 4$. We then give a  characterization of non-zero effective divisors  on a smooth sextic surface having $C.D=1$ and $2$ respectively, where $C$ is a smooth hyperplane section of the surface. In the third part,  we  first define initializedness and ACMness. Consequently, we document a result concerning an upperbound of global sections of such bundles on a smooth sextic.  After mentioning a brief remark about the characterization of Ulrich line bundles on a smooth sextic surface we move onto preparing a sufficient condition for ACMness of a divisor on a smooth surface of degree $d$. We end this section by documenting the definition and a useful result on $1$-connected divisor.\


In section \ref{S3}, we study necessary and sufficient condition for a line bundle given by a non-zero effective divisor  on a smooth sextic hypersurface to be initialized and ACM. In particular, towards the necessary condition we first study the possibilities in terms of $C.D$ and arithmetic genus $P_a(D)$. This is one of the most important steps towards proving Theorem \ref{T1} and a major part of this article is devoted to obtain a optimal picture of possibilities of arithmetic genus and intersection numbers through a detailed case by case analysis (cf. Theorem \ref{P3.1}). Here we mainly use the  tools showcased in section \ref{S2}, part $1$ and $2$. We then equip the already obtained necessary condition with some cohomological restrictions to obtain the only if part of the promised Theorem \ref{T1}.  Then  initializedness is established for the divisors satisfying any of those conditions. Before proving ACMness, using the techniques related to $1$-connectedness, we prove a proposition regarding vanishing of $h^1(\mathcal O_X(-D))$ for some divisors (see Proposition \ref{VAN}). This proposition coupled with techniques from section \ref{S2} enables us to prove ACMness for the divisors satisfying any of the conditions mentioned in the promised Theorem \ref{T1}. We mention that Section \ref{S2} third part plays the most crucial role in proving ACMness. We end this section by pointing out a concrete situation of the existence of a non-ACM line bundle on some smooth sextic surface.\\


 

\section*{Notations and Conventions}
All curves and surfaces are always smooth, projective over $\mathbb C$. We denote the canonical bundle of a smooth variety $X$ to be $K_X$. For a divisor (or the corresponding line bundle)  $\mathcal L$ on a variety $X$, we denote by $|\mathcal L|$  the linear system of $\mathcal L$ and denote the dual of a line bundle $\mathcal L$ by $\mathcal L^*$. In this context, $\text{deg}(\mathcal L)$ will always stand for the degree of the line bundle $\mathcal L$. For a divisor $\Delta$ on a curve $C$, $r(\Delta):= h^0(\mathcal O_C(\Delta))-1$ is the projective dimension of the complete linear system $|\Delta|$. A linear system on (a curve) $C$ having degree $d$ and projective dimension $r$ is denoted by $g^r_d$.\


We denote  $X^{(d)}$ by a smooth hypersurface of degree $d$ in $\mathbb P^3$. For a hyperplane section $C$ of $X^{(d)}$, we denote the class of it in $\text{Pic}(X^{(d)})$ by $H$. For an integer $t$, $H^{\otimes t}$ is  denoted as $\mathcal O_{X^{(d)}}(t)$. For a vector bundle $\mathcal E$ on $X^{(d)}$, we will write $\mathcal E \otimes \mathcal O_{X^{(d)}}(t) =\mathcal E(t)$.\ 

For a real number $x$, we denote by $[x]$ the greatest integer less than or equal to $x$.\

By S.E.S, we mean short exact sequence.\

\section{Technical Preliminaries}\label{S2}

In this section we will recall and establish few preliminary results that we need in the  next section. We begin by recalling the definitions of some important invariants of  smooth projective curves of genus  $g \geq 1$.\

\begin{definition}(gonality)
Let $C$ be a  curve of genus $g \geq 1$. The minimal degree of surjective morphisms from $C$ to $\mathbb P^1$ is called the gonality of $C$, and denoted by $\text{gon}(C)$ :
\begin{align*}
\text{gon}(C) :&= \text{min}\{\text{deg}(f) \mid f : C \to \mathbb P^1, \text{a surjective morphism}\}\\
&=  \text{min}\{d  \mid \exists \hspace{1mm} g^1_d \hspace{1mm} \operatorname{on} \hspace{1mm} C\}
\end{align*}
A curve of gonality $k$ is called a $k$-gonal curve. \
\end{definition}

The gonality is a classical invariant. We now mention another important invariant of smooth curves of genus $g \geq 1$, which is known as the Clifford index of a curve. Before recalling its definition, we note down a crucial motivating theorem:\

\begin{theorem}\label{CT}(Clifford's Theorem)
Let $\Delta$ be an  effective  divisor on a curve $C$. If $\Delta$ is special, i.e. $h^1(\mathcal O_C(\Delta)) \neq 0$, then $r(\Delta) \leq \frac{\text{deg}(\Delta)}{2}$. Furthermore, equality occurs if and only if either $\Delta =0$ or $\Delta = K_C$ or $C$ is hyperelliptic and $\Delta$ is a multiple of the unique $g^1_2$ on $C$.
\end{theorem}

\begin{proof}
See \cite{AG}, Theorem $5.4$.\
\end{proof}

In the context of Theorem \ref{CT}, we recall the definition of the Clifford index.\

\begin{definition}(Clifford index)\

$(i)$ Let $\Delta$ be an effective divisor on $C$. The Clifford index of $\Delta$ is defined by $\text{Cliff}(\Delta):= \text{deg}(\Delta)-2r(\Delta)$.\

$\Delta$ is said to contribute to the Clifford index if both $h^0(\mathcal O_C(\Delta)) \geq 2$ and $h^1(\mathcal O_C(\Delta)) \geq 2$ hold.\

$(ii)$ We define the Clifford index of a  curve of genus $g \geq 1$ as :\

$\text{Cliff}(C) := \text{min}\{\text{Cliff}(\Delta) \mid \Delta \hspace{1mm} \text{contributes  to the Clifford index}\}$.\

\end{definition}


The following classical result gives us the exact formula for the gonality and Clifford index of a smooth plane curve of degree $d \geq 5$ in terms of its degree $d$.\

\begin{lemma}\label{gon}[\cite{TH}, Lemma $0.3$]
Let $C$ be a smooth plane curve of degree $d \geq 5$. Then the gonality and the Clifford index of $C$ is determined only by the degree of $C$:
\begin{align*}
\text{gon}(C)=d-1, \text{Cliff}(C)=d-4.
\end{align*}
\end{lemma}


The following general remark helps estimate an upper bound of the dimension of global sections of line bundles on a  smooth hyperplane section of $X^{(d)}$. \

\begin{remark}\label{R1}
let $C$ be a smooth hyperplane section of $X^{(d)}$. Then  $|\mathcal O_C(1)|$ gives an embedding $C \hookrightarrow \mathbb P^2$. Since
$C$ is a plane curve of degree $d$, by Lemma \ref{gon}, $C$ is a $(d-1)$-gonal curve. Therefore, it can be established by induction that, if $\mathcal L$ is a line bundle on $C$ satisfying $h^0(\mathcal L) \geq 2$, then $\text{deg}(\mathcal L) \geq h^0(\mathcal L) +(d-1) -2$, i.e. $\text{deg}(\mathcal L) \geq h^0(\mathcal L)+d-3$. In particular, for $d=6$ one has, if $\mathcal L$ is a line bundle on $C$ satisfying $h^0(\mathcal L) \geq 2$, then  $\text{deg}(\mathcal L) \geq h^0(\mathcal L)+3$. Note that, for $d=5$, we recover \cite{W5}, Remark $2.1$.\
\end{remark}

We now mention an interesting result which deals with the existence(/non-existence) of certain line bundles on smooth plane curves of degree $d$. We will extensively use this result in the next section.\


\begin{proposition}\label{H1}
Let $g^r_n$ be a linear system on $C$(not necessarily very special). Write $g(C)=\frac{(d-1)(d-2)}{2}$.\

$(i)$ If $n >d(d-3)$, then $r=n-g$ (the non-special case).\

$(ii)$ If $n \leq d(d-3)$, then write $n=kd-e$ with $0 \leq k \leq d-3, 0 \leq e <d$, one has\

$\begin{cases} r \leq \frac{(k-1)(k+2)}{2}, &\text{ if } e >k+1 \\  r \leq \frac{k(k+3)}{2}-e, &\text{ if } e \leq k+1 \end{cases}$
\end{proposition}

\begin{proof}
See \cite{H6} Theorem $2.1$ or \cite{KC} Remark $1.3$.
\end{proof}

\begin{corollary}\label{wH1}
On a smooth plane sextic curve $C$, let $g^r_n$ be a linear system on $C$, then \


$(i)$ If $n=15$, then $r \leq 6$.\

$(ii)$ If $12 \leq n \leq 14$, then $r \leq 5$.\

$(iii)$ If $n=11$, then $r \leq 4$.\

$(iv)$ If $n=10$, then $r \leq 3$.\

$(v)$ If $7 \leq n \leq 9$, then $r \leq 2$.\





\end{corollary}

We end the discussion on curves by mentioning the Riemann-Roch Theorem for line bundles on curves. Let $\mathcal L$ be a line bundle on a  curve $C$ of genus $g$. Then the Riemann-Roch Theorem states that(see \cite{AG}, Chapter-$IV$, Theorem $1.3$) :\
\begin{align}\label{RRC}
h^0(C, \mathcal L)-h^0(C, \mathcal L^* \otimes K_C)= \text{deg}(\mathcal L)-g+1.\
\end{align}

Next we will discuss several useful results related to line bundles on $X^{(d)}$.\

From the piece of the following long exact cohomology sequence:
\begin{align*}
H^1(\mathbb P^3, \mathcal O_{\mathbb P^3}(m)) \to H^1(X^{(d)}, \mathcal O_{X^{(d)}}(m)) \to H^2(\mathbb P^3, \mathcal O_{\mathbb P^3}(m-d)) 
\end{align*}
and using the facts from the Cohomology of projective space (cf. \cite{AG}, Chapter-$3$, Theorem $5.1(b)$), we see that $H^1((X^{(d)}, \mathcal O_{X^{(d)}}(m)) =0, \forall m \in \mathbb Z$. More precisely, using Serre duality and above observation we have the following:
\begin{itemize}

\item  $H^0(\mathbb P^3, \mathcal O_{\mathbb P^3}(m)) \cong H^0(X^{(d)}, \mathcal O_{X^{(d)}}(m))$, for $m \leq d-1$.\

\item  $H^1(\mathbb P^3, \mathcal O_{\mathbb P^3}(m)) \cong H^1(X^{(d)}, \mathcal O_{X^{(d)}}(m)) =  0$,  $ \forall m \in  \mathbb Z$.\

\item  $H^2(X^{(d)}, \mathcal O_{X^{(d)}}(m)) \cong  H^0(\mathbb P^3, \mathcal O_{\mathbb P^3}(d-4-m))^*$, provided  $m \geq -3$.\

\end{itemize}

Note that, on $X^{(d)}$, one has by the adjunction formula $K_{X^{(d)}} \cong \mathcal O_{X^{(d)}}(d-4)$. One can compute its Euler characteristic as $\chi(\mathcal O_{X^{(d)}}) = 1 + \binom{d-1}{3}$. Let $C$ be a smooth hyperplane section of $X^{(d)}$. If $D$ is a non-zero effective divisor  on $X^{(d)}$, then 
\begin{itemize}

\item the arithmetic genus is given by: $P_a(D) = \frac{1}{2}D.(D+(d-4)C)+1$.\

\item the Riemann-Roch theorem for $\mathcal O_{X^{(d)}}(D)$ is given by: $\chi(\mathcal O_{X^{(d)}}(D))=  \frac{1}{2}D(D -(d-4)C)+ 1 + \binom{d-1}{3}$

\end{itemize}

Next, we mention a remark on sextic surfaces, which expresses the Euler characteristic of a divisor $D$ in terms of $C.D$ and arithmetic genus $P_a(D)$. This remark is exploited significantly in the next section.\

\begin{remark}
Consider, $X=X^{(6)}$. Let $D$ be a non-zero effective divisor on $X$. Then one has the following Riemann-Roch relation:
\begin{align}\label{K}
h^0(\mathcal O_X(D))-h^1(\mathcal O_X(D))+h^0(\mathcal O_X(2C-D))= \chi(\mathcal O_X(D)) =P_a(D)-2C.D+10.
\end{align}
\end{remark}

If $D$ is reduced and irreducible, then $P_a(D) \geq 0$. The following Proposition gives another sufficient condition  for $P_a(D) \geq 0$, which can be considered as a generalized version of \cite{W5}, Proposition $2.2$.\

\begin{proposition}\label{P1}
Let  $d \geq 4$. Let $D$ be a non-zero effective divisor on $X^{(d)}$. If $h^1(\mathcal O_{X^{(d)}}(-D)) =0$, then $P _a(D) \geq 0$.
\end{proposition}

\begin{proof}
Applying Riemann-Roch Theorem for line bundle $\mathcal O_{X^{(d)}}(-D)$ on the surface $X^{(d)}$, we have the following :

\begin{align*}
\chi(\mathcal O_{X^{(d)}}(-D))= & \frac{1}{2}(-D)(-D -(d-4)C)+ 1 + \binom{d-1}{3}\\
                      = & \frac{1}{2}D(D +(d-4)C)+ 1 + \binom{d-1}{3}\\
                      = & P_a(D) + \binom{d-1}{3}\
\end{align*}

From the ampleness of $C$, we have $C.D \geq 1 $, which implies $C.(-D) \leq -1$ and hence  $h^0(\mathcal O_{X^{(d)}}(-D)) =0$. This observation along with the Serre duality and first cohomology vanishing hypothesis one obtains :\

$h^0(\mathcal O_{X^{(d)}}(D +(d-4)C) = P_a(D) + \binom{d-1}{3}$.\

Now in order to obtain a lower bound of L.H.S we look at the following : \

$D \hookrightarrow X^{(d)}$ induces $\mathcal O_{X^{(d)}}(-D) \hookrightarrow \mathcal O_{X^{(d)}}$. Twisting with $\mathcal O_{X^{(d)}}(D + (d-4) C)$ yields  the following inclusion : $ \mathcal O_{X^{(d)}}((d-4)C) \hookrightarrow \mathcal O_{X^{(d)}}(D+(d-4)C)$. This means, $h^0(O_{X^{(d)}}((d-4)C)) \leq h^0(O_{X^{(d)}}(D+(d-4)C)$. Since $d-4 \leq d-1$, we have, $h^0(O_{X^{(d)}}((d-4)C)) \cong h^0(\mathcal O_{\mathbb P^3}(d-4))$ and therefore, one has $h^0(O_{X^{(d)}}(D+(d-4)C)) \geq \binom{d-1}{3}$, which in turn implies $P_a(D) \geq 0$.\\
\end{proof}



Let $D$ be a divisor on $X^{(d)}$ such that $C.D>0$. Consider the following  S.E.S:
\begin{align}\label{TES}
0 \to \mathcal O_{X^{(d)}}(-D) \to \mathcal O_{X^{(d)}}(C-D) \to \mathcal O_C(C-D) \to 0
\end{align}

In the context of S.E.S \ref{TES}, we  note down an easy remark regarding  an upperbound of $h^0(\mathcal O_{X^{(d)}}(C-D))$.\

\begin{remark}\label{L1}
Applying long exact cohomology sequence to the S.E.S \ref{TES} and  noting  $h^0(\mathcal O_{X^{(d)}}(-D))=0$, one obtains $h^0(\mathcal O_{X^{(d)}}(C-D)) \leq h^0(\mathcal O_C(C-D))$. By Remark \ref{R1}, we have $h^0(\mathcal O_{X^{(d)}}(C-D)) \leq 2$.\
\end{remark}
 





Let $X= X^{(6)}$. The following lemmas are useful  in showing the emptiness of certain linear systems in section \ref{S3}. These two lemmas (Lemma \ref{L2}, \ref{L3}) are adaptations of \cite{W5} Lemma $2.1, 2.2$ to the case of sextic surface.\

\begin{lemma}\label{L2}
Let $D$ be a divisor on $X$ satisfying $C.D=1$. Then $(a)\Rightarrow (b)$, where\

$(a)$ $h^0(\mathcal O_X(D)) >0$\

$(b)$ $h^0(\mathcal O_X(D))=1, h^0(\mathcal O_X(2C-D)) \geq 7$ and $D^2 \geq -4$.\

Moreover, under the assumption $C.D=1$ and $D^2\geq -4$, if $h^0(\mathcal O_X(2C-D)) \leq 7$, then  $h^0(\mathcal O_X(D)) >0$. 

\end{lemma}

\begin{proof}
Assume $(a)$. By the hypothesis, we may assume that $D$ is effective.
Since  $C$ is ample, $D$ is reduced and irreducible. This forces, $P_a(D) \geq 0$, and hence we have, $D^2 \geq -4$.\

Consider the S.E.S :\
\begin{align}\label{d1}
0\to \mathcal O_X(D-C) \to \mathcal O_X(D) \to \mathcal O_C(D) \to 0
\end{align}

Since $C.(D-C)<0$, by the ampleness of $C$, we have $h^0(\mathcal O_X(D-C)) =0$.  By Remark \ref{R1}, one sees that $h^0(\mathcal O_C(D)) = h^0(\mathcal O_X(D)) =1$. From Riemann-Roch Theorem we have:\
\begin{align}\label{e1}
h^0(\mathcal O_X(D))+h^0(\mathcal O_X(2C-D)) \geq \chi(\mathcal O_X(D)) \geq 8
\end{align}
Since  $h^0(\mathcal O_X(D)) =1$, we have $h^0(\mathcal O_X(2C-D)) \geq 7$. Conversely, under the assumption $C.D=1$ and $D^2 \geq -4$, we still have the inequality \ref{e1}, whence it follows that, if $h^0(\mathcal O_X(2C-D)) \leq 7$, then $h^0(\mathcal O_X(D)) >0$.

\end{proof}

\begin{lemma}\label{L3}
Let $D$ be a non-zero effective divisor on $X$ with $C.D =2$. If $D^2 \leq -8$, then one of the following cases occurs :\

$(i)$ There exists a curve $D_1$ on $X$ with $D=2D_1$, $D^2_1 \geq -4$ and $C.D_1=1$.\

$(ii)$ There exist curves $D_1$ and $D_2$ with $D=D_1 + D_2$, $D_1.D_2=0,D_i^2 =-4$ and $C.D_i=1$(for $i=1,2$).\

Conversely, Let  $D$ be a non-zero effective divisor on $X$ with $C.D =2$. If one of the following cases occurs :\

$(i)$ There exists a curve $D_1$ on $X$ with $D=2D_1$, $D^2_1 \geq -4$ , $C.D_1=1$ and $ P_a(D_1) \leq 1$.\

$(ii)$ There exist curves $D_1$ and $D_2$ with $D=D_1 + D_2$, $D_1.D_2=0,D_i^2 =-4$ and $C.D_i=1$(for $i=1,2$).\

Then, $D^2 \leq -8$

\end{lemma}

\begin{proof}
If $D^2 \leq -8$, then there exists a non-trivial effective decomposition $D=D_1+D_2$. As $C.D=2$, we must have $C.D_1=C.D_2=1$. Hence, by Lemma \ref{L2}, $D^2_i \geq -4(i=1,2)$. If $D_1=D_2$, then $D=2D_1$ and $D^2=4D^2_1 \geq -16$. If $D_1 \neq D_2$, then $D_1.D_2 \geq 0$. This forces,  $D^2 =-8$ and hence $D_1.D_2=0$ and $D^2_i=4(i=1,2)$. The converse assertion is clear.\

\end{proof}



Next we note down some well known facts about ACM bundles on $X=X^{(6)}$. Then we move onto preparing a proposition which will be crucial in proving the if part our main theorem. We start by recalling the definitions of the main objects concerned.\

\begin{definition}(ACMness)
We call a vector bundle $\mathcal E$ on $X$ an arithmetically Cohen-Macaulay(ACM) bundle if $H^1(\mathcal E(t))=0$ for all integers $t \in \mathbb Z$.\

\end{definition}

\begin{definition}(Initializedness)
For a sheaf $\mathcal F$ on $X$, we define the \textit{initial twist} as the integer $t$ such that $h^0(\mathcal F(t)) \neq 0$ and $h^0(\mathcal F(t-1))=0$. $\mathcal F$ is called initialized if its initial twist is $0$, i.e. if $h^0(\mathcal F)) >h^0(\mathcal F(-1))=0$.\

\end{definition}

For an ACM bundle $\mathcal E$ on $X$, we consider the graded module $H^0_{*}(\mathcal E) := \oplus_{t \in \mathbb Z}H^0(\mathcal E(t))$
over the homogeneous coordinate ring of $X$.  The following result gives an upperbound of  the minimal number of generators of it. It also deals with an upperbound of the global sections of an initialized and ACM line bundle on $X$ and an iff condition for a rank $r$ bundle on $X$ to be Ulrich.

\begin{proposition} (\cite{H3}, Theorem 3.1 and Corollary 3.5]). Let $\mathcal E$ be an ACM bundle
of rank $r$ on $X$, and let $\mu(\mathcal E)$ be the minimal number of generators of $H^0_{*}(\mathcal E)$. Then we get $\mu(\mathcal E) \leq 6r$. Moreover, if $\mathcal E$ is initialized, then $h^0(\mathcal E) \leq 6r$ and equality implies that $\mathcal E$ is an Ulrich bundle.
\end{proposition}


 Alternatively an Ulrich bundle of rank $r$ on $X$ can be characterized as an initialized and ACM bundle whose Hilbert polynomial is equal to $6r \binom{t+2}{2}$ (cf. \cite{CKM1}, Proposition $2.3$). Naturally one can ask where  the  Ulrich line bundles appears$\footnote{Note that, if there exists such a line bundle on the surface, then  it follows that the surface is linear determinantal. }$ in our main characterization Theorem (\ref{T1}) on initialized and ACM line bundles on $X$. Here we briefly remark a classification for Ulrich line bundles in the context of the promised Theorem \ref{T1}.\

\begin{remark}

We obtain the following numerical classification of Ulrich line bundles:\  

The line bundle $\mathcal O_X(D)$ defined by an effective divisor $D$ on $X$ is Ulrich iff the following conditions are satisfied:\
\begin{center}
 $P_a(D)=2C.D-4, C.D=15, h^0(\mathcal O_C(D))=6, h^0(\mathcal O_C(D-C))=0$ and $h^0(\mathcal O_C(D-2C))=0$.
 \end{center}
i.e., iff $C.D=15$ and satisfies the conditions mentioned in Theorem \ref{T1}$(iii)(a),(b), (c)$.\


\end{remark}

The following property of ACM line bundle will be useful in proving Proposition \ref{suff}.\

\begin{remark}\label{rACM}
A line bundle on $X$ is ACM if and only if the dual of it is ACM.
\end{remark}

Next we generalize Proposition $3.2$ of \cite{W5}  for $X^{(d)}$. This plays a crucial role in dealing with ACMness in the if part of Theorem \ref{T1}.\

\begin{proposition}\label{suff} 
Let $D$ be a non-zero effective divisor on $X^{(d)}$. Let $C$ be a smooth hyeperplane section of $X^{(d)}$ and  $k$ be a positive integer satisfying $C.D + d^2 <(k+4)d$. If $h^1(\mathcal O_{X^{(d)}}(tC-D)) =0$ for $0 \leq t \leq k$, then $\mathcal O_{X^{(d)}}(D)$ is ACM.
\end{proposition}

\begin{proof}


Note that, If  $n \geq k$, then  by hypothesis, $\text{deg}(\mathcal O_C(D+ (d-n-4)C)) = C.D +d(d-n-4) < 0$ and therefore, using Serre duality one has $h^1(\mathcal O_C((n+1)C-D)) =\footnote{\text{Observe that,} $K_C \cong \mathcal O_C(d-3)$.} h^0(\mathcal O_C(D+ (d-n-4)C))=0$. As $h^1(\mathcal O_{X^{(d)}}(tC-D)) =0$ for $0 \leq t \leq k$. Therefore, using the following S.E.S :\
\begin{center}
$0 \to \mathcal O_{X^{(d)}}(nC-D) \to \mathcal O_{X^{(d)}}((n+1)C-D) \to  \mathcal O_C((n+1)C-D) \to 0$,\
\end{center}
it can be seen that we have  by induction $h^1(O_{X^{(d)}}(nC-D)) =0,  \forall n \geq 1$. Next, note that, for $m \geq 0$, we have $h^0(\mathcal O_C(-mC-D))=0$ and from the  assumption, we have $h^1(\mathcal O_{X^{(d)}}(-D)) =0$. This means by the following S.E.S :\
\begin{center}
$0 \to \mathcal O_{X^{(d)}}(-(m+1)C-D)) \to O_{X^{(d)}}(-mC-D) \to O_C(-mC-D))  \to 0$,\

\end{center}

for any integer $m \geq 0$, we have $h^1(O_{X^{(d)}}(-mC-D))=0$ (by induction). Therefore, $\mathcal O_{X^{(d)}}(-D)$ is ACM, whence $\mathcal O_{X^{(d)}}(D)$ is  ACM by Remark \ref{rACM}. 

\end{proof}

\begin{corollary}\label{C1}
In particular for $X = X^{(6)}$, we obtain the following precise statement : Let $D$ be a non-zero effective divisor on $X$. Let $C$ be a smooth hyeperplane section of $X$ and  $k$ be a positive integer satisfying $C.D + 12 <6k$. If $h^1(\mathcal O_X(tC-D)) =0$ for $0 \leq t \leq k$, then $\mathcal O_X(D)$ is ACM.
\end{corollary}

We end this section by mentioning the concept of $1$-connectedness of a divisor, which will play an essential role in the proof of the vanishing of $h^1(\mathcal O_X(-D))$ for certain non-zero effective divisors in Proposition \ref{VAN}. We first recall the definition of $m$-connectedness of a divisor.\

\begin{definition}($m$-connected divisor)
Let $m$ be a positive integer. Then a non-zero effective divisor $D$ on a surface  is called $m$-connected, if $D_1.D_2 \geq m$, for each effective decomposition $D=D_1+D_2$.\

\end{definition}

The following Lemma gives us a necessary condition for a non-zero effective divisor $D$ to be $1$-connected.\

\begin{lemma}\label{CON}
If $D$ is $1$-connected effective divisor, then $h^0(\mathcal O_D)=1$.

\end{lemma}

\begin{proof}
See \cite{BPW}, Corollary $12.3$.\
\end{proof}

\begin{corollary}\label{CON1}

If a non-zero effective divisor $D$ on $X = X^{(6)}$ is $1$-connected, then by Lemma \ref{CON}, we have, $h^0(\mathcal O_D)=1$. Therefore, by applying long exact cohomology sequnce to the following S.E.S : 
\begin{center}
$0 \to \mathcal O_X(-D) \to \mathcal O_X \to \mathcal O_D \to 0$
\end{center}
one has, $h^1(\mathcal O_X(-D))=0$. This means, if $h^1(\mathcal O_X(-D)) \neq 0$, then $D$ is not $1$-connected. \

\end{corollary}


\section{Initialized and ACM line bundles on sextic surfaces}\label{S3}

In this section, we give a proof of Theorem \ref{T1},  which is divided into two subsections. In the first subsection, we investigate the necessary condition, and in the second subsection, we  show that if a non-zero effective divisor $D$ on $X= X^{(6)}$ satisfy  the necessary condition (i.e., condition $(i)$ and any one of the remaining conditions from $(ii)$ to $(viii)$ of the promised theorem \ref{T1}), then  the  line bundle $\mathcal O_X(D)$ given by it is initialized and ACM.\


\subsection{Necessary Condition}

Let $C$ be a smooth hyperplane section of $X = X^{(6)}$ as before. Let's assume that the line bundle $\mathcal O_X(D)$ given by a non-zero effective divisor $D$ on $X$ is initialized and ACM.  Then consider the following S.E.S :
\begin{align}\label{e2}
0 \to \mathcal O_X(C-D) \to \mathcal O_X(2C-D) \to \mathcal O_C(2C-D) \to 0.
\end{align}
Applying long exact cohomology sequence to the S.E.S \ref{e2} and using Remark \ref{rACM}, one obtains the following equality :
\begin{align}\label{N1}
h^0(\mathcal O_X(2C-D)) = h^0(\mathcal O_X(C-D))+ h^0(\mathcal O_C(2C-D)).
\end{align}
From Remark \ref{L1}, we already know that $h^0(\mathcal O_X(C-D)) \leq 2$. Note that, by remark \ref{R1} we have, $h^0(\mathcal O_C(2C-D)) \leq 8$. This means $h^0(\mathcal O_X(2C-D)) \leq 10$. In what follows we obtain   all possibilities of arithmetic genus $P_a(D)$ and corresponding range of $C.D$  for each such values of $h^0(\mathcal O_X(2C-D)) \in \{0,1,2,...,9,10\}$. In this pursuit our strategy is as follows: We  divide the numbers $0$ to $10$ into five distinct groups and use a unique technique for each one of them. To be more precise, the case $h^0(\mathcal O_X(2C-D))=0$ will be discussed as \textbf{CASE-$(i)$}, the cases $h^0(\mathcal O_X(2C-D))=1,2,3$ will be discussed as \textbf{CASE-$(ii)$}, the cases $h^0(\mathcal O_X(2C-D))=4,5,6$ will be discussed as \textbf{CASE-$(iii)$}, the cases $h^0(\mathcal O_X(2C-D))=7,8$ will be discussed as \textbf{CASE-$(iv)$} and finally it will be established that the remaining two cases can't occur.  \\

\underline{\textbf{CASE-$(i)$ $h^0(\mathcal O_X(2C-D))=0$.}}\

Here, we first analyze what happens  when for a degree $d \geq 5$ surface $X^{(d)}$, one has $h^0(\mathcal O_{X^{(d)}}(d-4)C-D)) =0$. We have from [\cite{H3}, Theorem 3.1 and Corollary 3.5], $1 \leq h^0(\mathcal O_{X^{(d)}}(D)) \leq d$. This can be rewritten as $h^0(\mathcal O_{X^{(d)}}(D)) =d-k$, where $0 \leq k \leq d-1$. By Riemann-Roch Theorem one has:\
 
\begin{center}
$\frac{1}{2}D^2 = d-k + \frac{(d-4)}{2} C.D -1 - \binom{d-1}{3}$.\

Therefore, $P_a(D) = d-k+(d-4)C.D-\binom{d-1}{3}$. \

\end{center}

From Proposition \ref{P1}, we have $P_a(D) \geq 0$, hence we have : 

\begin{center}
$\begin{cases} C.D \geq [\frac{1}{d-4}(\binom{d-1}{3}-1)] +1, &\text{ if } \frac{1}{d-4}(\binom{d-1}{3}-1)  \text{ is not an integer} \\   C.D \geq \frac{1}{d-4}(\binom{d-1}{3}-1), &\text{ otherwise }  \end{cases}$
\end{center}

We also have, $\chi(\mathcal O_{X^{(d)}}((d-3)C-D)) = h^0(\mathcal O_{X^{(d)}}(D)) + \frac{1}{2}d(d-3)-C.D$. By Remark \ref{rACM} and the assumptions of ACMness and initializedness, we have $C.D \leq h^0(\mathcal O_{X^{(d)}}(D))+ \frac{1}{2}d(d-3)$. 


 
  
This means, for $h^0(\mathcal O_{X^{(d)}}(D)) =1$, we have :
\begin{center}
     $P_a(D) =(d-4)C.D+1-\binom{d-1}{3}$ and \
     
   either  $[\frac{1}{d-4}(\binom{d-1}{3}-1)] +1 \leq C.D \leq 1+\frac{1}{2}d(d-3)$ or $\frac{1}{d-4}(\binom{d-1}{3}-1) \leq C.D \leq 1+\frac{1}{2}d(d-3)$.\ 
\end{center}

     Consider  $h^0(\mathcal O_{X^{(d)}}(D)) \geq 2$. Then by S.E.S \ref{d1} and assumptions of initializedness and ACMness, one has $h^0(\mathcal O_{X^{(d)}}(D))= h^0(\mathcal O_C(D)) = d-k \geq 2$. By Remark \ref{R1}, we have, $C.D \geq 2d-k-3$. This means, for $h^0(\mathcal O_{X^{(d)}}(D)) \geq 2 $, we have :
  
\begin{center}   
       $P_a(D) =(d-4)C.D+ d-k-\binom{d-1}{3}$ and $2d-k-3 \leq C.D \leq \frac{1}{2}d^2-\frac{1}{2}d-k$. 
 \end{center}
       
       Note that, for $d=5$, we recover \cite{W5}, Lemma $4.3$. In particular, when $d=6$, we obtain the following :
\begin{center}
for $h^0(\mathcal O_X(D)) =1$ : $P_a(D)= 2C.D-9, 5 \leq C.D \leq 10$ and\

for $h^0(\mathcal O_X(D)) \geq 2$: $P_a(D)= 2C.D- 4-k, 9-k \leq C.D \leq 15-k$\
        
\end{center}
        
         We complete the analysis of this case (for $d=6$) by mentioning that some possibilities  can't occur.\

Note that, for each $k$, one has from the assumption of ACMness and initializedness, $h^0(\mathcal O_C(D)) = h^0(\mathcal O_X(D)) =6-k$. In view of that it's not very difficult to see the following :\

\begin{itemize}

\item The cases $k=0, 9 \leq C.D \leq 10$ and $k=1, C.D=8$ are not possible by Theorem \ref{CT}.\

\item  The cases $k=0, C.D=11$, $k=1, C.D =9$ and $k=2, C.D=7$ are not possible. In all of these cases one obtains $\text{Cliff}(\mathcal O_C(D)) = 1$, a contradiction to Lemma \ref{gon}.\

\item The cases $k=1, C.D=10$ and $k=2, 8 \leq C.D \leq 9$ are not possible by Corollary \ref{wH1}.\

\end{itemize}

We represent the remaining possibilities in Table \ref{A0}.\\

\begin{table}\label{A0}
    \centering
    \begin{tabular}{c|c}
     $P_a(D)$   &  $C.D$\\
     \hline \hline
       $2CD-9$  & $5 \leq C.D \leq 10$\\
       $2CD-8$  & $5 \leq C.D \leq 11$\\
       $2CD-7$  & $6 \leq C.D \leq 12$\\
       $2CD-6$  & $10 \leq C.D \leq 13$\\
       $2CD-5$  & $11 \leq C.D \leq 14$\\
       $2CD-4$  & $12 \leq C.D \leq 15$\\
      
    \end{tabular}
    \caption{CASE-$(i)$}
    \label{tab:my_label}
\end{table}

Note that, since in each of the remaining cases we have $|2C-D| \neq \emptyset$, from \ref{N1}, one sees that in each such cases two situations arise, which are either $h^0(\mathcal O_C(2C-D)) \neq 0$ or $|C-D| \neq \emptyset$.\\

\underline{\textbf{CASE-$(ii)$ $ h^0(\mathcal O_X(2C-D))= 3-m$, where $m \in \{0,1,2\}$}}.\



Let's assume $h^0(\mathcal O_C(2C-D)) \neq 0$, then one must have $C.(2C-D) \geq 0$. This means we have $C.D \leq 12$. If $C.D=12$, then $\mathcal O_X(D) \cong \mathcal O_X(2C)$, a contradiction to the assumption that $\mathcal O_X(D)$ is initialized. This forces $C.D \leq 11$.  From the assumption of ACMness, initializedness and Remark \ref{R1}, we have $h^0(\mathcal O_C(D))=h^0(\mathcal O_X(D)) \leq 8$. Therefore, one can write $h^0(\mathcal O_C(D))=h^0(\mathcal O_X(D))=8-r$, where $r \in \{0,1,2,3,4,5,6\}$. For such a fixed $r$, one can write $C.D = 11-r+s$ (by Remark \ref{R1}), where $s \in \{0,...r\}$. For fixed $m, r$ and $s$ in the admissible range, one obtains by assumption and  relation \ref{K}, 
\begin{align}
P_a(D)=23-m-3r+2s. 
\end{align}
Note that, for any $m \in \{0,1,2\}$, the following possibilities can't occur:\
\begin{itemize}
\item The cases $r=0$, $1 \leq r \leq 2$ with $0 \leq s \leq 1$ and $r=3$ with $s=0$ are not possible by Theorem \ref{CT}.\

\item The cases $r=2$ with $s=2$, $r=3$ with $s=1$ and $r=4$ with $s=0$ are not possible. In all three cases one obtains $\text{Cliff}(\mathcal O_C(D)) = 1$, a contradiction to Lemma \ref{gon}.\

\item The cases $r=3$ with $s=2$ and $r=4$ with $1 \leq s \leq 2$ are not possible by Corollary \ref{wH1}.\
\end{itemize}

Next we consider  the case,  $h^0(\mathcal O_C(D))=h^0(\mathcal O_X(D))=1$. Since $C.D \leq 11$, from  relation \ref{K}, one obtains
\begin{align}
P_a(D)=2C.D-6-m \leq 16-m. 
\end{align}

Note that, for any $m \in \{0,1,2\}$, the following possibilities can't occur:\
\begin{itemize}
\item The case $P_a(D)= 16-m$ is not possible. Because in that case $C.(2C-D)=1$ implies by Lemma \ref{L2}, one must have $(2C-D)^2 \geq -4$, a contradiction.\

\item The case $P_a(D)=14-m$ is not possible. Note that, as $C.(2C-D)=2$, if we take $\Gamma \in |2C-D|$, then by Lemma \ref{L3},in each case one must obtain that there exists a curve $\Gamma_1$ on $X$ with $C.\Gamma_1 =1$, $\Gamma^2_1 \geq -4$ such that $\Gamma=2.\Gamma_1$. This means for $m=2$, $4\Gamma_1^2 =-14$, a contradiction. For $m=1$ this gives us $\Gamma_1^2=-3$ and hence $P_a(\Gamma_1)$ is a fraction, a contradiction. Finally for $m=0$, this forces $4.\Gamma^2_1=-10$, a contradiction.
\end{itemize}

Note that, since $\mathcal O_X(D)$ is ACM, by Proposition \ref{P1}, we have $P_a(D) \geq 0$. In this situation (i.e. when $h^0(\mathcal O_C(D))=1$), we are left with the following possibilities:
\begin{align*}
&\text{for} \thickspace m=0, P_a(D)=2i \thickspace \text{and} \thickspace C.D=3+i, \thickspace \text{where} \thickspace i \in \{0,1,2,3,4,5,6\}\\
& \text{for}\thickspace m=1, P_a(D)=2i-1\thickspace \text{and} \thickspace C.D=3+i, \thickspace \text{where} \thickspace i \in \{1,2,3,4,5,6\}\\
& \text{for} \thickspace m=2, P_a(D)=2i \thickspace \text{and} \thickspace C.D=4+i, \thickspace \text{where} \thickspace i \in \{0,1,2,3,4,5\}\
\end{align*}

Next let's consider the case $|C-D| \neq \emptyset$. From  the S.E.S \ref{TES} and the assumption of ACMness one has $h^0(\mathcal O_C(C-D)) \neq 0$. This forces $C.D \leq 6$.  If $C.D=6$, then $\mathcal O_X(D) \cong \mathcal O_X(C)$, a contradiction to the assumption that $\mathcal O_X(D)$ is initialized. This forces $C.D \leq 5$.  From the assumption of ACMness, initializedness and Remark \ref{R1}, we have $h^0(\mathcal O_C(D))=h^0(\mathcal O_X(D)) \leq 2$. Note that, the case $h^0(\mathcal O_C(D))=2, C.D=5$ is not possible. This is because, in that case $C.(C-D)=1$, which implies by Lemma \ref{L2}, $(C-D)^2 \geq -4$, a contradiction (for each values of $m$). Therefore, we are left with the case $h^0(\mathcal O_C(D))=1$. Note that, as before in this situation one finds using $C.D \leq 5$ and relation \ref{K},
\begin{align}
P_a(D)=2C.D-6-m \leq 4-m. 
\end{align} 
Note that, for any $m \in \{0,1,2\}$ the case $P_a(D)= 4-m$ is not possible. Because in that case $C.(C-D)=1$, which implies by Lemma \ref{L2}, one must have $(C-D)^2 \geq -4$, a contradiction.\

Again as before Since $\mathcal O_X(D)$ is ACM, by Proposition \ref{P1}, we have $P_a(D) \geq 0$. In this situation (i.e. when $h^0(\mathcal O_C(D))=1$), we are left with the following possibilities:
\begin{align*}
&\text{for} \thickspace m=0, P_a(D)=2i \thickspace \text{and} \thickspace C.D=3+i, \thickspace \text{where} \thickspace i \in \{0,1\}\\
& \text{for}\thickspace m=1, P_a(D)=2i-1\thickspace \text{and} \thickspace C.D=3+i, \thickspace \text{where} \thickspace i=1\\
& \text{for} \thickspace m=2, P_a(D)=2i \thickspace \text{and} \thickspace C.D=4+i, \thickspace \text{where} \thickspace i=0\
\end{align*}

We represent the possibilities that are obtained from the above analysis in the  Tables \ref{A1} (for $m=2$), \ref{A2} (for $m=1$), and \ref{A3} (for $m=0$). \\

\begin{table}\label{A1}
    \centering
    \begin{tabular}{c|c}
     $P_a(D)$   &  $C.D$\\
     \hline \hline
       
       $2CD-8$  & $4 \leq C.D \leq 9$\\
       $2CD-7$  & $5 \leq C.D \leq 11$\\
       $2CD-6$  & $6 \leq C.D \leq 11$\\
       $2CD-5$  & $10 \leq C.D \leq 11$\\
       $2CD-4$  & $ C.D = 11$\\
      
    \end{tabular}
    \caption{CASE-$(ii)$, $m=2$}
    \label{tab:my_label}
\end{table}

\begin{table}\label{A2}
    \centering
    \begin{tabular}{c|c}
     $P_a(D)$   &  $C.D$\\
     \hline \hline

       $2CD-7$  & $4 \leq C.D \leq 9$\\
       $2CD-6$  & $5 \leq C.D \leq 11$\\
       $2CD-5$  & $6 \leq C.D \leq 11$\\
       $2CD-4$  & $10 \leq C.D \leq 11$\\
       $2C.D-3$ & $ C.D = 11$\\
      
    \end{tabular}
    \caption{CASE-$(ii)$, $m=1$}
    \label{tab:my_label}
\end{table}

\begin{table}\label{A3}
    \centering
    \begin{tabular}{c|c}
     $P_a(D)$   &  $C.D$\\
     \hline \hline

       $2CD-6$  & $3 \leq C.D \leq 9$\\
       $2CD-5$  & $5 \leq C.D \leq 11$\\
       $2CD-4$  & $6 \leq C.D \leq 11$\\
       $2CD-3$  & $10 \leq C.D \leq 11$\\
       $2CD-2$  & $ C.D = 11$\\
      
    \end{tabular}
    \caption{CASE-$(ii)$, $m=0$}
    \label{tab:my_label}
\end{table}

\underline{\textbf{CASE-$(iii)$ $h^0(\mathcal O_X(2C-D))=6-m$, where $m \in \{0,1,2\}$}.}\

Note that, in this situation we have from  Lemma \ref{L1} and  relation \ref{N1} that $h^0(\mathcal O_C(2C-D)) \geq 4-m \geq 2$. From Remark \ref{R1}, we see that $C.(2C-D) \geq h^0(\mathcal O_C(2C-D)) +3 \geq 7-m$. This means in this case we have, $1 \leq C.D \leq 5+m$. From the assumption of ACMness, initializedness and Remark \ref{R1}, we have $h^0(\mathcal O_C(D))=h^0(\mathcal O_X(D)) \leq 2+m$. Therefore, one can write $h^0(\mathcal O_C(D))= h^0(\mathcal O_X(D))=2+m-r$, where $r \in \{0,....,m\}$. For  fixed $m, r$, one can write $C.D = 5+m-r+s$ (by Remark \ref{R1}), where $s \in \{0,...r\}$. For fixed $m, r$ and $s$ in the admissible range, one obtains by assumption and the relation \ref{K}, 
\begin{align}
P_a(D)=8-3r+2m+2s. 
\end{align}
Note that, for  $m=2$, the possibilitiy $r=0$ can't occur. Indeed, in this case, one obtains $\text{Cliff}(\mathcal O_C(D))=1$, a contradiction to Lemma \ref{gon}.\ 

Next we consider  the case  $h^0(\mathcal O_C(D))= h^0(\mathcal O_X(D))=1$. Since $C.D \leq 5+m$, from the relation \ref{K}, one obtains
\begin{align}
P_a(D)=2C.D-3-m \leq 7+m. 
\end{align}

Note that, for  $m \in \{0,1,2\}$, $P_a(D) \neq 7+m, 5+m$ and for $m \in \{1,2\}$, $P_a(D) \neq 3+m$. These can be realized by the following arguments.\
\begin{itemize}
\item Note that, for  $P_a(D)=7+m$, one can show that for $m \in \{1,2\}$, $|C-D| = \emptyset$. Indeed this can be seen for $m=1$ from initializedness, for $m=2$ from the observation $C.(C-D) <0$ and the S.E.S \ref{TES} (and using the assumptions of ACMness and initializedness). This forces $h^0(\mathcal O_C(2C-D)) = 6-m$ (by \ref{N1}), a contradiction by Remark \ref{R1}. For $m=0$, by Remark \ref{R1} and the S.E.S \ref{TES}, we must have $h^0(\mathcal O_X(C-D)) \leq 1$. This forces $h^0(\mathcal O_C(2C-D)) \geq 5$ (by \ref{N1}), a contradiction by Remark \ref{R1}.\

\item Note that, for $P_a(D)=5+m$, we see that for $m=2$, one has by initializedness $|C-D| = \emptyset$, which implies $h^0(\mathcal O_C(2C-D)) =4$ (by \ref{N1}), a contradiction by Remark \ref{R1}. For $m=1$, we have by Remark \ref{R1} and the S.E.S \ref{TES}, $h^0(\mathcal O_X(C-D)) \leq 1$ and hence $h^0(\mathcal O_C(2C-D)) \geq 4$ (by \ref{N1}). This forces $\text{Cliff}(\mathcal O_C(2C-D)) \leq 1$, a contradiction by Lemma \ref{gon}. For $m=0$, again by Remark \ref{R1} and the S.E.S \ref{TES}, we have $h^0(\mathcal O_X(C-D)) \leq 1$. This implies $h^0(\mathcal O_C(2C-D)) \geq 5$ (by \ref{N1}), a contradiction by Theorem \ref{CT}.\ 

\item Note that, for $P_a(D)=3+m$, we observe that for $m=2$, by Lemma \ref{L2}, one obtains $|C-D| = \emptyset$ and hence $h^0(\mathcal O_C(2C-D)) = 4$ (by \ref{N1}). This forces $\text{Cliff}(\mathcal O_C(2C-D)) = 1$, a contradiction by Lemma \ref{gon}. For $m=1$, one has by Remark \ref{R1} and the S.E.S \ref{TES}, $h^0(\mathcal O_X(C-D)) \leq 1$ and hence $h^0(\mathcal O_C(2C-D)) \geq 4$ (by \ref{N1}), a contradiction by Corollary \ref{wH1}.\
\end{itemize}

Observe that, Since $\mathcal O_X(D)$ is ACM, by Proposition \ref{P1}, we have $P_a(D) \geq 0$. In this situation (i.e. when $h^0(\mathcal O_C(D))=1$), we are left with the following possibilities:
\begin{align*}
&\text{for} \thickspace m=0, P_a(D)=2i-1 \thickspace \text{and} \thickspace C.D=1+i, \thickspace \text{where} \thickspace i \in \{1,2\}\\
& \text{for}\thickspace m=1, P_a(D)=2i \thickspace \text{and} \thickspace C.D=2+i, \thickspace \text{where} \thickspace i \in \{0,1\}\\
& \text{for} \thickspace m=2, P_a(D)=2i-1 \thickspace \text{and} \thickspace C.D=2+i, \thickspace \text{where} \thickspace i \in \{1,2\}\
\end{align*}

We can compactly rewrite the possibilities that arise from  above calculations in  tabular form.(Table \ref{A4} for $m=2$, Table \ref{A5} for $m=1$, Table \ref{A6} for $m=0$)\\


\begin{table}\label{A4}
    \centering
    \begin{tabular}{c|c}
     $P_a(D)$   &  $C.D$\\
     \hline \hline

       $2CD-5$  & $3 \leq C.D \leq 4$\\
       $2CD-4$  & $5 \leq C.D \leq 7$\\
       $2CD-3$  & $6 \leq C.D \leq 7$\\
       
    \end{tabular}
    \caption{CASE-$(iii)$, $m=2$}
    \label{tab:my_label}
\end{table}



\begin{table}\label{A5}
    \centering
    \begin{tabular}{c|c}
     $P_a(D)$   &  $C.D$\\
     \hline \hline

       $2CD-4$  & $2 \leq C.D \leq 3$\\
       $2CD-3$  & $5 \leq C.D \leq 6$\\
       $2CD-2$  & $C.D=6$\\
       
    \end{tabular}
    \caption{CASE-$(iii)$, $m=1$}
    \label{tab:my_label}
\end{table}



\begin{table}\label{A6}
    \centering
    \begin{tabular}{c|c}
     $P_a(D)$   &  $C.D$\\
     \hline \hline

       $2CD-3$  & $2 \leq C.D \leq 3$\\
       $2CD-2$  & $C.D=5$\\
       
    \end{tabular}
    \caption{CASE-$(iii)$, $m=0$}
    \label{tab:my_label}
\end{table}

\underline{\textbf{CASE-$(iv)$ $h^0(\mathcal O_X(2C-D))=8-m$, where $m \in \{0,1\}$}.}\

Consider the situation $h^0(\mathcal O_X(C-D))=2$. In this case by the assumption of initializedness, Remark \ref{R1} and the S.E.S \ref{TES}, we have, $h^0(\mathcal O_C(C-D))=2$. This means by Remark \ref{R1}, $C.D=1\footnote{\text{In this case it is easy to see that, we have}, $h^0(\mathcal O_C(D))=h^0(\mathcal O_X(D))=1$}$ and hence by relation \ref{K},
\begin{align}
P_a(D)=1-m.\
\end{align}

Next let's consider the situation $h^0(\mathcal O_X(C-D))=2-j$, where $j\in \{1,2\}$. In this case by relation \ref{N1}, we obtain $h^0(\mathcal O_C(2C-D))=6-m+j \geq 2$. By Remark \ref{R1}, this implies, $C.D \leq 3+m-j$. Therefore for a fixed $m$ and $j$ in the admissible range we can write $C.D =3+m-j-l$, where $l \in \{0,1,2\}\footnote{\text{In this representation we must have} $C.D \geq 1$. \text{ Therefore, under this notation the following situations don't occur}\
\begin{itemize}
\item For $m=0,j=1$, $l$ can't be $2$ and for $m=0,j=2$, $l$ can't be $1$ or $2$.\
\item For $m=1,j=2$, $l$ can't be $2$.
\end{itemize}}$. Also note that, in this situation one must have by Remark \ref{R1} and the assumptions of ACMness and initializedness, $h^0(\mathcal O_X(D))=h^0(\mathcal O_C(D))=1$. Thus from relation \ref{K}, one obtains:
\begin{align*}
P_a(D)=5+m-2j-2l.
\end{align*}

Therefore, the possibilities that arise out of the above analysis are as follows:
\begin{align*}
& \text{for} \thickspace m=1, \thickspace P_a(D)=2C.D-2 \thickspace \text{and} \thickspace 1 \leq C.D \leq 3. \\
& \text{for} \thickspace m=0, \thickspace P_a(D)=2C.D-1 \thickspace \text{and} \thickspace 1 \leq C.D \leq 2.\\
\end{align*}

Next we will show that the cases \underline{\textbf{$h^0(\mathcal O_X(2C-D))=10-m$, where $m \in \{0,1\}$}} are not possible. Note that, in this situation by ampleness of $C$, \ref{N1} and  Remark \ref{R1}, we must have  $h^0(\mathcal O_X(C-D)) \neq 0, 1-m$. Now consider the situation $h^0(\mathcal O_X(C-D))=2$. In this case by the assumption of initializedness, Remark \ref{R1} and the S.E.S \ref{TES}, we have, $h^0(\mathcal O_C(C-D))=2$. This means by Remark \ref{R1}, $C.D=1$ and hence by relation \ref{K},
\begin{align}
P_a(D)=3-m.\
\end{align}
Next Let's consider the case $h^0(\mathcal O_X(C-D))=1$ (Observe that this can only happen when $m=1$). Then from  \ref{N1} and Remark \ref{R1}, one obtains $C.D=1$ (and hence $h^0(\mathcal O_C(D))=1$). Therefore, from relation \ref{K}, one obtains $P_a(D)=2$. Thus, the possibilities that arise out of the above analysis are as follows:
\begin{align*}
& \text{for} \thickspace m=1, \thickspace P_a(D)=2C.D \thickspace \text{and} \thickspace  C.D =1. \\
& \text{for} \thickspace m=0, \thickspace P_a(D)=2C.D+1 \thickspace \text{and} \thickspace  C.D=1.
\end{align*}

These in particular tell us that in each case we have $D^2=2(1-m)$. Therefore, by Riemann-Roch Theorem we obtain:\
\begin{align*}
h^0(\mathcal O_X(C+D))+h^0(\mathcal O_X(C-D)) =\chi(\mathcal O_X(C+D))=9-m
\end{align*}

This means by Remark \ref{L1}, $h^0(\mathcal O_X(D+C)) \geq 7-m$. Applying the long exact cohomology sequence to the  following S.E.S:

\begin{align}
0 \to \mathcal O_X(D) \to \mathcal O_X(D+C) \to \mathcal O_C(D+C) \to 0
\end{align} 

and using the assumption of ACMness of initializedness, one obtains $h^0(\mathcal O_C(C+D)) \geq 6-m \geq 2$. This forces by Remark \ref{R1}, $m\geq 2$, a contradiction.\\

Now that we have analyzed all the cases and obtained all the possibilities of $P_a(D)$ and $C.D$, we further move onto improving the scenario of all possibilities. To be more precise, in what follows next we show that certain cases that obtained as possibilities in previous analysis can't occur.\\

We start by showing that the case $P_a(D)=2C.D-1$ is not possible. Note that, in this case we had $C.D=1+i$, where $i \in \{0,1\}$. In this situation it is easy to see that $h^0(\mathcal O_C(D))=h^0(\mathcal O_X(D))=1, h^0(\mathcal O_X(C-D)) \leq 2-i$. Therefore, by relations \ref{N1} and \ref{K}, we end up having  $h^0(\mathcal O_C(2C-D)) \geq 6+i$. This amounts to $\text{Cliff}(\mathcal O_C(2C-D)) \leq 1$, a contradiction to Lemma \ref{gon}.\\


Next we will establish that in case $P_a(D)=2C.D-2$ except $C.D=1$, the other possibilities can't occur.\

Consider the possibility $C.D=11$. Note that, in this situation one has, $h^0(\mathcal O_X(D))=h^0(\mathcal O_C(D)) \leq 6$. Otherwise from \ref{RRC} we have $h^0(\mathcal O_C(3C-D)) \geq 5$, a contradiction by Remark \ref{R1}. Then from  relation \ref{K}, one obtains $h^0(\mathcal O_X(2C-D)) \geq 2$, a contradiction by Lemma \ref{L2}.\

Consider the possibilities $C.D=5+i$, where $i \in \{0.1\}$. By Remark \ref{R1}, we have $h^0(\mathcal O_C(D)) \leq 2+i$. Also it's easy to see that $h^0(\mathcal O_X(C-D)) \leq 1-i \footnote{ \text{For} i=1, \text{this follows from initializedness and for} i=0 \text{this follows by Remark \ref{R1} }}$. Then from  relations \ref{K} and \ref{N1}, we end up having $h^0(\mathcal O_C(2C-D))\geq 5$, a contradiction by Remark \ref{R1}.\

Consider the possibilities $C.D=2+i$, where $i \in \{0.1\}$. By Remark \ref{R1}, we have $h^0(\mathcal O_C(D))= h^0(\mathcal O_X(D))=1, h^0(\mathcal O_X(C-D)) \leq 1$.  Therefore, by relations \ref{K} and \ref{N1}, we obtain $h^0(\mathcal O_C(2C-D))\geq 6$, a contradiction by Theorem \ref{CT}.\\

Next we will show that the case $P_a(D)=2C.D-3$, can't occur.\

Consider the possibility $C.D=11$. Note that, in this case we have $h^0(\mathcal O_C(D)) \geq 7$ (Observe that, if $|2C-D| \neq \emptyset$, then this can be realized from Lemma \ref{L2} and the S.E.S \ref{d1} and if $|2C-D| = \emptyset$, then it follows from  relation \ref{K} and the short exact sequence \ref{d1}), a contradiction by Theorem \ref{CT}.\

Consider the possibility $C.D=10$.  Note that, in this situation one has, $h^0(\mathcal O_X(D))=h^0(\mathcal O_C(D)) \leq 6$. Otherwise from \ref{RRC} we have $h^0(\mathcal O_C(3C-D)) \geq 6$, a contradiction by Remark \ref{R1}. Also by Theorem \ref{CT}, $h^0(\mathcal O_C(D)) \neq 6$. Then from  relations \ref{K} and \ref{N1}, we  have $h^0(\mathcal O_C(2C-D))\geq 2$, a contradiction by Remark \ref{R1}.\

Next let's consider the possibilities $C.D =6+i$,  where $i \in \{0,1\}$. In this situation by Remark \ref{R1}, we have $h^0(\mathcal O_C(D))\leq 3+i$. This coupled with relation \ref{K} gives us $h^0(\mathcal O_X(2C-D)) \geq 4-i$. It's not difficult to see that in both cases we have $h^0(\mathcal O_X(C-D))=0$. This forces by \ref{N1}, $h^0(\mathcal O_C(2C-D)) \geq 4-i$, a contradiction by Remark \ref{R1}.\

Now consider the possibilities $C.D=2+i$, where $i \in \{0,1,3\}$. By Remark \ref{R1}, we have the following : for $i=3$, $h^0(\mathcal O_X(D))=h^0(\mathcal O_C(D)) \leq 2$, for $i \in \{0,1\}$, $h^0(\mathcal O_C(D))=h^0(\mathcal O_X(D))=1$ and for all values of $i$, $h^0(\mathcal O_X(C-D)) \leq 1$. Therefore, from relations \ref{K} and  \ref{N1}, we have for $i=3$, $h^0(\mathcal O_C(2C-D)) \geq 4$ and  for $i \in \{0,1\}$, $h^0(\mathcal O_C(2C-D)) \geq 5$. For $i \in \{1,3\}$,  this means $\text{Cliff}(\mathcal O_C(2C-D)) \leq 1$, a contradiction to Lemma \ref{gon}. By the same argument for $i=0$, $h^0(\mathcal O_C(2C-D))$ can't be $\geq 6$, whence we have $h^0(\mathcal O_C(2C-D)) =5$, a contradiction to Corollary \ref{wH1}.\\

Next we show that for $P_a(D)=2C.D-4$ the possibilities $C.D=11,10,9,8,7,3$ can't occur.\

For $C.D=11$, one can argue as in the previous case (i.e., $P_a(D)=2C.D-3$ and $C.D=11$) to obtain in any case $h^0(\mathcal O_C(D)) \geq 6$. This forces $\text{Cliff}(\mathcal O_C(D)) \leq 1$, a contradiction to Lemma \ref{gon}.\

Consider the possibilities $C.D=9+i$, where $i \in \{0,1\}$. Then as before by \ref{RRC} and Remark \ref{R1}, one must obtain $h^0(\mathcal O_C(D)) \leq 6$. Note that, by Theorem \ref{CT}, $h^0(\mathcal O_C(D)) \neq 6$. For $i=1$, by Corollary \ref{wH1}, $h^0(\mathcal O_C(D)) \neq 5$ and for $i=0$, by Lemma \ref{gon}, $h^0(\mathcal O_C(D)) \neq 5$. This yields by relations \ref{K} and  \ref{N1}, $h^0(\mathcal O_C(2C-D)) \geq 2$, a contradiction by Remark \ref{R1}.\

Consider the possibilities $C.D=7+i$, where $i \in \{0,1\}$. Then by Remark \ref{R1}, one has $h^0(\mathcal O_C(D)) \leq 4+i$. This implies by relations \ref{K}, \ref{N1} and Remark \ref{R1}, $h^0(\mathcal O_C(2C-D))=h^0(\mathcal O_X(2C-D))=2-i$. This forces $h^0(\mathcal O_C(D))=4+i$, a contradiction by Corollary \ref{wH1}.\

For $C.D=3$, by Remark \ref{R1}, we have $h^0(\mathcal O_C(D))= h^0(\mathcal O_X(D))=1, h^0(\mathcal O_X(C-D)) \leq 1$.  Thus, by relations \ref{K} and \ref{N1}, we end up having $h^0(\mathcal O_C(2C-D))\geq 4$, a contradiction by Corollary \ref{wH1}.\\

Finally  we show that for $P_a(D)=2C.D-5$, the cases $C.D=8,9$ can't occur.\

Here in both the cases using  relations \ref{K},  \ref{N1} and Remark \ref{R1}, one can show that the situation $h^0(\mathcal O_X(D)) \leq 3$ isn't possible. This forces $h^0(\mathcal O_C(D)) \geq 4$, a contradiction by Corollary \ref{wH1}.\\




We summarize the content of the above discussion in the following Theorem.\

\begin{theorem}\label{P3.1}
With the notations as before, If $\mathcal O_X(D)$ is ACM and initialized, then the following conditions are satisfied:\

$(a)$ either $2C.D-9 \leq P_a(D) \leq 2C.D-4$ or $P_a(D)=2C.D-2$.\

$(b)$ If $P_a(D) =2C.D-9$, then $5 \leq C.D \leq 10$.\

$(c)$ If $P_a(D) =2C.D-8$, then $4 \leq C.D \leq 11$.\

$(d)$ If $P_a(D) =2C.D-7$, then $4 \leq C.D \leq 12$.\

$(e)$ If $P_a(D) =2C.D-6$, then $3 \leq C.D \leq 13$.\

$(f)$ If $P_a(D) =2C.D-5$, then $3 \leq C.D \leq 7$ or $10 \leq C.D \leq 14$.\

$(g)$ If $P_a(D) =2C.D-4$, then $ C.D =2$ or $5 \leq C.D \leq 6$ or $12 \leq C.D \leq 15$.\


$(i)$ If $P_a(D) =2C.D-2$, then $ C.D =1$ .\

\end{theorem}

With Theorem \ref{P3.1} at hand we are now ready prove the only if direction of the promised Theorem \ref{T1}.\

$\textit{Proof of the only if part of the Theorem \ref{T1}}$

\begin{proof}

Note that, $(i)$, $(ii)$ and $(iii)(a)$ follows from Theorem \ref{P3.1}. To establish $(iii)(b),(c)$ we argue as follows :\

For $(iii)(b)$  first part Consider $C.D=6$ or  $12 \leq C.D \leq 15$.  Then applying long exact cohomology sequence to the following  S.E.S:\

\begin{align}\label{s1}
0 \to \mathcal O_X(D-2C) \to \mathcal O_X(D-C) \to \mathcal O_C(D-C) \to 0
\end{align}

and from the assumption of ACMness and initializedness, we have $h^0(\mathcal O_C(D-C)) =0$.\

For the second part let $12 \leq C.D \leq 15$. Then consider the following S.E.S :

\begin{align}\label{m1}
0 \to \mathcal O_X(D-3C) \to \mathcal O_X(D-2C) \to \mathcal O_C(D-2C) \to 0
\end{align}

Taking long exact cohomology sequence to the S.E.S \ref{m1} and using the assumption of ACMness and initializedness, one has $h^0(\mathcal O_C(D-2C)) =0$.\

For $(iii)(c)$ first part, note that, for $C.D=12$, one must have $|2C-D|=\emptyset$ (Otherwise $\mathcal O_X(D) \cong \mathcal O_X(2C)$, a contradiction). Then from relation \ref{N1} the desired result follows. For the second part Consider $C.D=15$. Since $C.(2C-D) <0$, we have $h^0(\mathcal O_C(2C-D))=0$ and hence $h^0(\mathcal O_X(2C-D))=0$ (by \ref{N1}). Then from  relation \ref{K}, one has $h^0(\mathcal O_X(D)) = 6$. This means $h^0(\mathcal O_C(D)) = 6$  from the S.E.S \ref{d1}.\\

Next, Consider the case $(iv)$. In this case observe that $(iv)(a)$ follows from Theorem \ref{P3.1} and $(iv)(b)$ first part follows by the same argument as that of $(iii)(b)$ first part.\



For $(iv)(b)$ second part if  $C.D=11$, then by Lemma \ref{L2} and if $C.D=12$, then by the same argument as in $(iii)(c)$ first part, one obtains $|2C-D| = \emptyset$. This means by relation \ref{N1}, $h^0(\mathcal O_C(2C-D))=0$.\

For $(iv)(c)$  first part, note that, by initializedness, one has $|C-D| = \emptyset$. Then from relations \ref{K}, \ref{N1} and S.E.S \ref{d1}, one must have $h^0(\mathcal O_C(D))+h^0(\mathcal O_C(2C-D)) =5$, whence the claim follows. The second part follows from ACMness and the S.E.S \ref{TES}\

For $(iv)(d)$  let $12 \leq C.D \leq 14$. In this case for $13 \leq C.D \leq 14$ from $C.(2C-D) <0$ and relation \ref{N1} it's easy to see that, we have $h^0(\mathcal O_C(2C-D))=h^0(\mathcal O_X(2C-D))=0$. Note that we already have this for $C.D=12$. Then from  relation \ref{K} and the S.E.S \ref{d1}, one obtains $h^0(\mathcal O_C(D)) = 5$.\\

Next, Consider the case $(v)$. In this case observe that $(v)(a)$ follows from Theorem \ref{P3.1}, $(v)(b)$ first part follows by the same argument as that of $(iii)(b)$ first part and $(v)(b)$ second part $11 \leq C.D \leq 12$ follows by the same argument as $(iv)(b)$ second part.\

For $(v)(b)$ second part for $C.D=10$, note that, by relation \ref{N1} and by $C.(C-D) < 0$ it's enough to show that $|2C-D| = \emptyset$. This follows from Lemma \ref{L3}.\



Note that, $(v)(c)$ first part follows from the fact $h^0(\mathcal O_X(D))=1$, relation \ref{N1}, \ref{K} and from S.E.S \ref{TES}. The argument for $(v)(c)$ second part is same as that of $(iv)(c)$ second part.\

It's clear that $(v)(d)$, $(v)(e)$ and $(v)(f)$ first part follows from  S.E.S \ref{d1}, relations \ref{K}, \ref{N1} and observing that in each such case $|C-D| = \emptyset\footnote{\text{For} $C.D=7,$ \text{this follows from} $C.(C-D) <0,$ \text{for} $C.D=6,$ \text{this follows from initializedness and for} $C.D=5,$ \text{this follows from Lemma \ref{L2}.}}$.\

For $(v)(f)$ second part let $11 \leq C.D \leq 13$. Since $C.(2C-D) <0$ for $C.D=13$, we have $h^0(\mathcal O_C(2C-D))=0$. We have the same vanishing for $C.D=11,12$ from $(v)(b)$ second part. Then from relations \ref{K}, \ref{N1}, one has in all  three situations $h^0(\mathcal O_X(D)) = 4$. This means $h^0(\mathcal O_C(D)) = 4$  from short exact sequence \ref{d1}.\\

Now consider the case $(vi)$. Note that, here $(vi)(a)$ follows from Theorem \ref{P3.1}, $(vi)(b)$ follows by the same argument as that of $(v)(b)$. Note here that for $C.D=5$,   $|C-D| = \emptyset$ follows from Lemma \ref{L2}  and for $C.D=6$, $|C-D| = \emptyset$ follows from initializedness. Then $(vi)(c),(d),(e), (f)$ follows from the relations \ref{K}, \ref{N1} and the S.E.S \ref{d1}.\

For $(vi)(g)$,  one obtains for $C.D=4$, $|C-D| = \emptyset$ by Lemma \ref{L3}. Note that, here we have by Remark \ref{R1} and S.E.S \ref{d1}, $h^0(\mathcal O_C(D))=h^0(\mathcal O_X(D))=1$. Then the desired result follows from relations \ref{K} and \ref{N1}.\


Since for $C.D=4,6$, we have $|C-D| =\emptyset$, argument for $(vi)(h)$ first part is same as that of $(iv)(c)$ second part. For $(vi)(h)$ second part observe that $(vi)(b)$ second part, relations \ref{K}, \ref{N1} and S.E.S \ref{d1} give  us the desired condition.\\

Consider the case $(vii)$. As before $(vii)(a)$ follows from Theorem \ref{P3.1}. It can be shown by Lemma \ref{L2} that for $C.D=5$, one has $|C-D|= \emptyset$. Note that, we have the same result for $C.D=6$ by initializedness. Then $(vii)(b)$ first part follows by the same argument as that of $(vi)(c)$.  $(vii)(b)$ second part is clear from initializedness, ACMness, S.E.S \ref{d1} and the relation \ref{K}.\

 For $C.D=4$, it is not difficult to see that $|C-D|= \emptyset$ (by lemma \ref{L3}). Then $(vii)(c)$ first part follows (all three cases) by the same argument as that of $(vi)(h)$ first part. For $(vii)(c)$ second part, note that, if $h^0(\mathcal O_C(2C-D)) \geq 2$, then by relations \ref{K}, \ref{N1}, one obtains $h^0(\mathcal O_X(D))=h^0(\mathcal O_C(D))=0$, a contradiction. \

 $(vii)(d)$ first part follows by noting that $h^0(\mathcal O_X(D))=1$ and using relations \ref{K}, \ref{N1}. To establish $(vii)(d)$ second part we argue as before that in both the situations one has $h^0(\mathcal O_C(2C-D))=0$. This forces by relations \ref{K}, \ref{N1}, S.E.S \ref{d1} the desired condition.\\

Finally let's consider the case $(viii)$. Note that $(viii)(a)$ first part follows from  Theorem \ref{P3.1}.  For $C.D=5$ one can show by Lemma \ref{L2} that $|C-D| = \emptyset$. The same observation can be made for $C.D=6$ by initializedness. Then it can be seen that  $(viii)(a)$ second part follows by relation \ref{K}, \ref{N1} and the S.E.S \ref{d1}. The last condition $(viii)(b)$ follows by the same argument as $(vii)(c)$ first part. \\

\end{proof}

\subsection{Sufficient Condition}

In this section we give a proof of the "if" part of the Theorem \ref{T1}. Let
 $X \subset \mathbb P^3$ be a smooth sextic hypersurface. Let $H$ be the hyperplane class of $X$ and $C \in |H|$ be a smooth member as before. Let $D$ be a non-zero effective divisor on $X$ of arithmetic genus $P_a(D)$. Let $D$ satisfy the condition $(i)$ and any one of the remaining conditions (i.e., from $(ii)$ to $(viii)$ with the restrictions mentioned under the condition) of the promised Theorem \ref{T1}. Our first aim to show that $\mathcal O_X(D)$ is initialized.\\
 
$\textit{Proof of the  if part of the Theorem \ref{T1}}$\

\begin{proof}

\textit{\textbf{Initializedness}}

Note that, for any value of arithmetic genus of $D$,  if $C.D \leq 5$, then initializedness follows from the observation  $C.(D-C) <0$.\

Note that, irrespective of the value of the arithmetic genus of $D$, one has for $9 \leq C.D \leq 11$, $C.(D-2C)<0$ and hence $h^0(\mathcal O_X(D-2C))=0$. Therefore, in these cases it's enough to have $h^0(\mathcal O_C(D-C))=0$ (because in that case initializedness follows from by applying long exact cohomology sequnce to the S.E.S \ref{s1}). Now for $9 \leq C.D \leq 11$, this is there by assumptions $(iv)(b),(v)(b), (vi)(b)$ and by the observation that in remaining cases $h^0(\mathcal O_C(D))<3$. For $C.D=6$, if $|D-C| \neq \emptyset$, then we have $\mathcal O_X(D)) \cong \mathcal O_X(C)$, a contradiction (as then $P_a(D)=10$).\

Again for any value of arithmetic genus of $D$, for $C.D =7$, $|D-C| = \emptyset$ can be obtained by applying Lemma \ref{L2}, to the divisor $D-C$.\

Similarly for the cases $(vi),(vii),(viii)$, If $C.D =8$, $|D-C| = \emptyset$ can be obtained by applying Lemma \ref{L3}, to the divisor $D-C$. For the case $(v)$, if $C,D=8$, then initiazedness can be realized by the same argument as that of the situations $9 \leq C.D \leq 11$.\

Finally consider the cases,  $12 \leq C.D \leq 15$. Note that, in each such cases we  have $C.(D-3C) <0$ and $h^0(\mathcal O_C(D-2C))=0$ ( for $12 \leq C.D \leq 15$, this follows either by the assumption $(iii)(b)$ second part or by noting that $h^0(\mathcal O_C(D)) <6$). Then from the  S.E.S \ref{m1}, one has $h^0(\mathcal O_X(D-2C))=0$. Finally applying long exact cohomology sequence to the S.E.S \ref{s1} and using $h^0(\mathcal O_C(D-C))=0$, we have $h^0(\mathcal O_X(D-C))=0$ ( Note here that, $h^0(\mathcal O_C(D-C))=0$ is either there by assumptions $(iii)(b),(iv)(b),(v)(b),(vi)(b)$ or can be seen by noting that in the remaining case one has $h^0(\mathcal O_C(D)) <3$).\\

To prove  ACMness, we need a Proposition which deals with the vanishing of $h^1(\mathcal O_X(-D))$ for effective divisors with low values of $C.D$. This Proposition uses a consequence of a non-zero effective divisor $D$ being not $1$-connected.\

\begin{proposition}\label{VAN}
Let $D$ be a non-zero effective divisor on a smooth sextic surface $X \subset \mathbb P^3$ such that $h^1(\mathcal O_X(D))=0$ and $h^0(\mathcal O_X(D-C))=0$. Then we have, $h^1(\mathcal O_X(-D))=0$, if the following conditions are satisfied:

$(i)$ $2C.D-7 \leq P_a(D) \leq   2C.D-5$, $C.D=5$.\

$(ii)$ $P_a(D)=2C.D-4$, $C.D=6$.\



\end{proposition}

\begin{proof}


We suitably adapt the techniques used in Lemma $2.3$, \cite{W5} according to our situation. In what follows, We note down an important consequence of such $D$ being not $1$-connected. We work under the  assumption that $D^2 =2C.D-2k$ for some positive integer $k\footnote{\text{Observe that, when} $P_a(D)=2C.D-r$, \text{then we always have this expression for some integer} $k$.}$.\

Assume that $D$ is not $1$-connected. Let $D=D_1+D_2$ be a non-trivial effective decomposition with $D_1.D_2 \leq 0$. Since $D^2=D.D_1+D.D_2$, we may assume that $D.D_1 =D^2-D.D_2 \leq \frac{D^2}{2}$. Then we have, $\frac{D^2}{2} \leq D.D_2=D_2^2 +D_1.D_2 \leq D_2^2$. Note that,

\begin{align*}
4\chi(\mathcal O_X(D_2)) &= 4(\frac{1}{2}D_2(D_2-2C)+11)\\
                         &= 2D_2^2 -4C.D_2+44
\end{align*}
and 
\begin{align*}
4h^0(\mathcal O_X(D)) &= 4\chi(\mathcal O_X(D)) -4 h^0(\mathcal O_X(2C-D))\\
                      &= 2D^2-4 C.D +44 -4 h^0(\mathcal O_X(2C-D))
\end{align*}
Therefore,
\begin{align*}
4\chi(\mathcal O_X(D_2)) - 4h^0(\mathcal O_X(D)) & =(2D_2^2 -4C.D_2)-(2D^2-4C.D) + 4 h^0(\mathcal O_X(2C-D))\\
& \geq D^2-4C.D_2 +4k +4h^0(\mathcal O_X(2C-D))\\
& =2C.D-2k-4C.D_2 +4k +4h^0(\mathcal O_X(2C-D))\\
& =2C.D_1-2C.D_2 +2k +4h^0(\mathcal O_X(2C-D)) 
\end{align*}

This in turn gives us:

\begin{align}\label{3.6}
\chi(\mathcal O_X(D_2)) \geq h^0(\mathcal O_X(D)) +h^0(\mathcal O_X(2C-D)) +\frac{1}{2}(C.D_1-C.D_2 +k) \
\end{align}



With inequality \ref{3.6} at hand, we now analyze what happens in each case if we start with the assumption  $h^1(\mathcal O_X(-D)) \neq 0$. Note that,  by Corollary \ref{CON1}, this assumption implies $D$ is not $1$-connected, which enables us to apply the inequality \ref{3.6}. Therefore, to conclude the Proposition, it is enough to obtain contradiction in both the cases.\

$(i)$ Let's consider \textbf{$P_a(D)=2C.D-r, C.D=5$},  where $r \in \{5,6,7\}$\

Note that, in this situation, one has  $D^2=8-2r$ and hence $D^2 =2C.D-2k$ for $k=r+1$. From \ref{3.6}, there exist divisors $D_1,D_2$ on $X$ with $D=D_1+D_2$ such that  $\chi(\mathcal O_X(D_2)) \geq 10-r +\frac{1}{2}(C.D_1-C.D_2 +k)$. Assume that $C.D_1 < C.D_2$. Since, $C.D_1 \geq 1$, we have $C.D_2 \geq 2$. In this situation the only possibilities are either $C.D_1= 1, C.D_2=4$ or $C.D_1=2, C.D_2=3$.  Consider the case $C.D_1=1, C.D_2=4$. Since, $C.D_1-C.D_2 +k=r-2$, one obtains $h^0(\mathcal O_X(D_2)) +h^0(\mathcal O_X(2C-D_2)) \geq \chi(\mathcal O_X(D_2)) \geq 9-\frac{r}{2}$. Applying long exact cohomology sequence to the S.E.S \ref{e2} and noting $h^0(\mathcal O_X(D_2)) \leq h^0(\mathcal O_C(D_2)) \leq 1$, we have $1 +h^0(\mathcal O_C(2C-D_2))+h^0(\mathcal O_X(C-D_2))  \geq 9-\frac{r}{2} $. Since $C.D_2$ is  $4$, by Corollary \ref{wH1}, we have $h^0(\mathcal O_C(2C-D_2)) \leq 3$ and by Remark \ref{R1}, we have $h^0(\mathcal O_X(C-D_2))\leq 1$. This forces $5 \geq 1 +h^0(\mathcal O_C(2C-D_2))+h^0(\mathcal O_X(C-D_2)) \geq 9-\frac{r}{2}$, a contradiction. Consider the case $C.D_1=2, C.D_2=3$. Since, $C.D_1-C.D_2 +k=r$, as before, one obtains $h^0(\mathcal O_C(D_2)) +h^0(\mathcal O_C(2C-D_2))+h^0(\mathcal O_X(C-D_2)) \geq 10-\frac{r}{2}$. Since $C.D_2$ is  $3$, by Corollary \ref{wH1}, we have $h^0(\mathcal O_C(2C-D_2)) \leq 3$ and by Remark \ref{R1}, we have $h^0(\mathcal O_X(C-D_2))\leq 1$. This forces $5 \geq 1 +h^0(\mathcal O_C(2C-D_2))+h^0(\mathcal O_X(C-D_2)) \geq 10-\frac{r}{2}$, a contradiction. Now consider the case $C.D_2 \leq C.D_1$. In this situation the only possibilities are either $C.D_1= 4, C.D_2=1$ or $C.D_1=3, C.D_2=2$.  Again as before, we have, $1 +h^0(\mathcal O_C(2C-D_2))+h^0(\mathcal O_X(C-D_2)) \geq 10-r +\frac{1}{2}(C.D_1-C.D_2 +k)$. Consider the case, $C.D_2=2, C.D_1 =3$, then we have $C.D_1-C.D_2 +k =r+2$ and therefore, one obtains $1+h^0(\mathcal O_X(C-D_2)) +h^0(\mathcal O_C(2C-D_2))  \geq 11-\frac{r}{2}$. By Remark \ref{R1}, $h^0(\mathcal O_X(C-D_2)) \leq 1$ and by Corollary \ref{wH1}, we have $h^0(\mathcal O_C(2C-D_2)) \leq 4$. But this then means $6 \geq 1 +h^0(\mathcal O_C(2C-D_2))+h^0(\mathcal O_X(C-D_2)) \geq 11-\frac{r}{2}$, a contradiction. Finally consider the case, $C.D_2=1, C.D_1 =4$, then we have $C.D_1-C.D_2 +k=r+4$ and therefore, one obtains $1 +h^0(\mathcal O_C(2C-D_2))+h^0(\mathcal O_X(C-D_2)) \geq 12-\frac{r}{2}$. By Remark \ref{R1}, $ h^0(\mathcal O_X(C-D_2)) \leq 2$ and by Corollary \ref{wH1}, we have $h^0(\mathcal O_C(2C-D_2)) \leq 5$. This means $8 \geq 1 +h^0(\mathcal O_C(2C-D_2))+h^0(\mathcal O_X(C-D_2)) \geq 12-\frac{r}{2}$, a contradiction.\\

To establish $(ii)$  Let's consider \textbf{$P_a(D)=2C.D-r, C.D=6$}, where $r=4$. Note that, in this situation, one has  $D^2=10-2r$ and hence $D^2 =2C.D-2k$,  for $k=r+1$. From \ref{3.6}, there exist divisors $D_1,D_2$ on $X$ with $D=D_1+D_2$ such that  $\chi(\mathcal O_X(D_2)) \geq 10-r +\frac{1}{2}(C.D_1-C.D_2 +k)$. Assume that $C.D_1 < C.D_2$. Since, $C.D_1 \geq 1$, we have $C.D_2 \geq 2$. In this situation the only possibilities are either $C.D_1= 1, C.D_2=5$ or $C.D_1=2, C.D_2=4$.  Consider the case $C.D_1=1, C.D_2=5$. Since, $C.D_1-C.D_2 +k=r-3$, one obtains $h^0(\mathcal O_X(D_2)) +h^0(\mathcal O_X(2C-D_2)) \geq \chi(\mathcal O_X(D_2)) \geq 8.5-\frac{r}{2}$. Applying long exact cohomology sequence to the S.E.S \ref{e2} and noting $h^0(\mathcal O_X(D_2)) \leq h^0(\mathcal O_C(D_2)) \leq 2$, we have $2 +h^0(\mathcal O_C(2C-D_2))+h^0(\mathcal O_X(C-D_2))  \geq 8.5-\frac{r}{2} $. Since $C.D_2$ is  $5$, by Corollary \ref{wH1}, we have $h^0(\mathcal O_C(2C-D_2)) \leq 3$ and by Remark \ref{R1}, we have $h^0(\mathcal O_X(C-D_2))\leq 1$. This forces $6 \geq 2 +h^0(\mathcal O_C(2C-D_2))+h^0(\mathcal O_X(C-D_2)) \geq 8.5-\frac{r}{2}$, a contradiction. Consider the case $C.D_1=2, C.D_2=4$. Since, $C.D_1-C.D_2 +k=r-1$, as before, one obtains $h^0(\mathcal O_C(D_2)) +h^0(\mathcal O_C(2C-D_2))+h^0(\mathcal O_X(C-D_2)) \geq 9.5-\frac{r}{2}$. Since $C.D_2$ is  $4$, by Corollary \ref{wH1}, we have $h^0(\mathcal O_C(2C-D_2)) \leq 3$ and by Remark \ref{R1}, we have $h^0(\mathcal O_X(C-D_2))\leq 1$. This forces $5 \geq 1 +h^0(\mathcal O_C(2C-D_2))+h^0(\mathcal O_X(C-D_2)) \geq 9.5-\frac{r}{2}$, a contradiction. Now consider the case $C.D_2 \leq C.D_1$.  In this situation the only possibilities are either $C.D_1= 5, C.D_2=1$ or $C.D_1=4, C.D_2=2$ or $C.D_1=3, C.D_2=3$ .  Again as before, we have, $1 +h^0(\mathcal O_C(2C-D_2))+h^0(\mathcal O_X(C-D_2)) \geq 10-r +\frac{1}{2}(C.D_1-C.D_2 +k)$. Consider the case, $C.D_2=2, C.D_1 =4$, then we have $C.D_1-C.D_2 +k =r+3$ and therefore, one obtains $1+h^0(\mathcal O_X(C-D_2)) +h^0(\mathcal O_C(2C-D_2))  \geq 11.5-\frac{r}{2}$. By Remark \ref{R1}, $h^0(\mathcal O_X(C-D_2)) \leq 1$ and by Corollary \ref{wH1}, we have $h^0(\mathcal O_C(2C-D_2)) \leq 4$. But this then means $6 \geq 1 +h^0(\mathcal O_C(2C-D_2))+h^0(\mathcal O_X(C-D_2)) \geq 11.5-\frac{r}{2}$, a contradiction. Next consider the case, $C.D_2=1, C.D_1 =5$, then we have $C.D_1-C.D_2 +k=r+5$ and therefore, one obtains $1 +h^0(\mathcal O_C(2C-D_2))+h^0(\mathcal O_X(C-D_2)) \geq 12.5-\frac{r}{2}$. By Remark \ref{R1}, $ h^0(\mathcal O_X(C-D_2)) \leq 2$ and by Corollary \ref{wH1}, we have $h^0(\mathcal O_C(2C-D_2)) \leq 5$. This means $8 \geq 1 +h^0(\mathcal O_C(2C-D_2))+h^0(\mathcal O_X(C-D_2)) \geq 12.5-\frac{r}{2}$, a contradiction. Finally consider the case, $C.D_2=3, C.D_1 =3$, then we have $C.D_1-C.D_2 +k=r+1$ and therefore, one obtains $1 +h^0(\mathcal O_C(2C-D_2))+h^0(\mathcal O_X(C-D_2)) \geq 10.5-\frac{r}{2}$. By Remark \ref{R1}, $ h^0(\mathcal O_X(C-D_2)) \leq 1$ and by Corollary \ref{wH1}, we have $h^0(\mathcal O_C(2C-D_2)) \leq 3$. This means $5 \geq 1 +h^0(\mathcal O_C(2C-D_2))+h^0(\mathcal O_X(C-D_2)) \geq 10.5-\frac{r}{2}$, a contradiction.\\

\begin{remark}
From the strategy of the above proof it can be seen that for case-$(ii)$, if we consider $r=5$, then we get contradiction in each case except the case $C.D_1=1, C.D_2=5$. In this situation it can be deduced that $D_i^2=0$ for $i \in \{1,2\}$ and $D_1.D_2=0$. For Case-$(i)$, if we consider $r=8$, then we get contradiction except in two situations. The first such situation is when  $C.D_1=1, C.D_2=4$. In this case again one obtains $D_i^2=-4$ for $i \in \{1,2\}$ and $D_1.D_2=0$. In second situation we have $C.D_1=4, C.D_2=1$. In this case we have $D_i^2=-4$ for $i \in \{1,2\}$ and $D_1.D_2=0$. In this situation it can also be noticed that $D_2$ is an ACM line bundle of type $(ii)$ (i.e., satisfies the condition $(ii)$ of the promised Theorem \ref{T1} and is ACM and initialized).

\end{remark}

\end{proof}

With Proposition \ref{VAN} at hand we are now ready to prove the ACMness part of the if direction of the promised Theorem \ref{T1}. Before proceeding further we would like to briefly describe our strategy of proving ACMness (In the context of Corollary \ref{C1}) in $6$ steps.\

\underline{\textbf{Step-$(I)$}} : We use Remark \ref{R1} (and occasionally Corollary \ref{wH1}) to obtain an upperbound of $h^0(\mathcal O_C(D))$ and hence an upperbound of $h^0(\mathcal O_X(D))$ by S.E.S \ref{d1} (Here in some situations for convenience we may first start with an upperbound of $h^0(\mathcal  O_C(2C-D))$ and use the same technique as mentioned here to obtain precise cohomological values). By relation \ref{K} this gives us a lowerbound of $h^0(\mathcal O_X(2C-D))$ and hence an lowerbound of $h^0(\mathcal O_C(2C-D))$ (by S.E.S \ref{e2}). Then by assumption or Corollary \ref{wH1} or by Remark \ref{R1}, we obtain precise values of $h^0(\mathcal O_X(D)) =h^0(\mathcal O_C(D)), h^0(\mathcal O_C(2C-D)), h^0(\mathcal O_X(2C-D))$. This values when substituted to the relation \ref{K} give us $h^1(\mathcal O_X(D))=h^1(\mathcal O_X(2C-D))=0$. We here mention that this is the most important step towards proving ACMness.\

\underline{\textbf{Step-$(II)$}} : We apply long exact cohomology sequence to the S.E.S \ref{d1}. Then from the initializedness and step-$(I)$, we have $h^1(\mathcal O_X(3C-D))=h^1(\mathcal O_X(D-C))=0$ (Note that, from step-$(I)$, we will always have $h^0(\mathcal O_X(D))=h^0(\mathcal O_C(D))$).\

\underline{\textbf{Step-$(III)$}} : We apply long exact cohomology sequence to the S.E.S \ref{s1}. Then from step-$(II)$ and  $h^0(\mathcal O_C(D-C))=0$, one obtains $h^1(\mathcal O_X(4C-D))=h^1(\mathcal O_X(D-2C))=0$.\

\underline{\textbf{Step-$(IV)$}} :  We apply long exact cohomology sequence to the S.E.S \ref{m1}. Then from step-$(III)$ and  $h^0(\mathcal O_C(D-2C))=0$, one obtains $h^1(\mathcal O_X(5C-D))=h^1(\mathcal O_X(D-3C))=0$.\

\underline{\textbf{Step-$(V)$}} :  We apply long exact cohomology sequence to the S.E.S \ref{e2}. Then using the informations obtained in step-$(I)$, we obtain $h^1(\mathcal O_X(C-D))=0$.\

\underline{\textbf{Step-$(VI)$}} : If we have $h^0(\mathcal O_C(C-D))=h^0(\mathcal O_X(C-D))$, then we apply long exact cohomology sequence to the S.E.S \ref{TES} and use step-$(V)$ to see $h^1(\mathcal O_X(-D))=0$ or else the vanishing of $h^1(\mathcal O_X(-D))$ will follow from Proposition \ref{VAN}.\\

Also note that, for establishing ACMness, by Corollary \ref{C1}, irrespective of the value of the arithmetic genus of $D$ we have the following :

\begin{itemize}

\item For $C.D \leq 5$, it's  enough to carry out steps-$(I),(II),(V),(VI)$.\

\item For $6 \leq C.D \leq 11$, it's  enough to carry out all steps except step-$(IV)$.\

\item  For $12 \leq C.D \leq 15$, we need to carry out all the $6$ steps.\

\end{itemize}

\textit{\textbf{ ACMness}
 }

Consider  \underline{\textbf{$P_a(D)=2C.D-2.$}}

Let $C.D=1$. Then $D$ is reduced and irreducible and therefore, step-$(VI)$ is done. By Remark \ref{R1}, initializedness and S.E.S \ref{d1}, \ref{TES}, we have  $h^0(\mathcal O_C(D))=h^0(\mathcal O_X(D))=1, h^0(\mathcal O_X(C-D)) \leq 2$. Now from the relation \ref{K}, one has $h^0(\mathcal O_X(2C-D)) \geq 7$.  From the S.E.S \ref{e2}, we have $h^0(\mathcal O_X(2C-D)) \leq h^0(\mathcal O_C(2C-D))+h^0(\mathcal O_X(C-D))$.  If $h^0(\mathcal O_C(2C-D))$ is $ \geq 6$, then $\text{Cliff}(\mathcal O_C(2C-D)) \leq 1$, a contradiction to Lemma \ref{gon}. This means $h^0(\mathcal O_C(2C-D))=5$ and hence $h^0(\mathcal O_X(2C-D))=7$, $h^0(\mathcal O_X(C-D))=h^0(\mathcal O_C(C-D))=2$. Then from relation \ref{K}, we see that step-$(I)$ is completed and hence step-$(V)$ is also done. It's clear that step-$(II)$ follows from step-$(I)$ and initializedness. \\

Consider  \underline{\textbf{$P_a(D)=2C.D-4.$}}

Let $C.D=2$. Again as before by Remark \ref{R1}, initializedness and S.E.S \ref{d1}, \ref{TES},  we have $h^0(\mathcal O_C(D))=h^0(\mathcal O_X(D))=1, h^0(\mathcal O_X(C-D)) \leq 1$. By the same argument as in previous case Step-$(I)$, one can obtain $h^0(\mathcal O_C(2C-D)) \geq 4$. By Corollary \ref{wH1}, this means  $h^0(\mathcal O_C(2C-D)) = 4$.  From S.E.S \ref{e2}, one sees that, $ h^0(\mathcal O_X(C-D))=h^0(\mathcal O_C(C-D))=1, h^0(\mathcal O_X(2C-D))=5$.  Then from relation \ref{K}, step-$(I)$ is completed. Steps-$(II),(V)$ is done by the same arguments as in the previous case. Since from step-$(I)$ we have $h^0(\mathcal O_X(C-D))=h^0(\mathcal O_C(C-D))$, step-$(VI)$ is done using step-$(V)$. \\

Let $C.D=5$.  Note that, here we have $h^0(\mathcal O_C(D)) \leq 2, h^0(\mathcal O_C(C-D))\leq 1$ (by Remark \ref{R1}). By the same arguments as in previous case step-$(I)$, one can obtain $h^0(\mathcal O_C(2C-D)) \geq 3$. By Corollary \ref{wH1}, this means  $h^0(\mathcal O_C(2C-D)) = 3$.  From S.E.S \ref{e2}, one sees that, $ h^0(\mathcal O_X(C-D))=h^0(\mathcal O_C(C-D))=1, h^0(\mathcal O_X(2C-D))=4$. This forces by relation \ref{K} and S.E.S \ref{TES}, $h^0(\mathcal O_X(D))=h^0(\mathcal O_C(D))=2$. Then from relation \ref{K}, step-$(I)$ is completed. The arguments for steps-$(II),(V),(VI)$ are same as that of previous case (i.e, the case $C.D=2$).\\

Let $C.D=6$.  From Remark \ref{R1}, we have $h^0(\mathcal O_C(D)) \leq 3$. Note that, in this case $|C-D| =\emptyset$ (Otherwise $\mathcal O_X(D) \cong \mathcal O_X(C)$, a contradiction). Then by the same argument as in previous case step-$(I)$, one can obtain $h^0(\mathcal O_C(2C-D)) \geq 3$. By Remark \ref{R1}, this means  $h^0(\mathcal O_C(2C-D)) =3$.  From S.E.S \ref{e2}, one sees that, $ h^0(\mathcal O_X(2C-D))= h^0(\mathcal O_C(2C-D))= 3$. This forces from relation \ref{K} and S.E.S \ref{TES}, $h^0(\mathcal O_X(D))=h^0(\mathcal O_C(D))=3$. Then from relation \ref{K}, step-$(I)$ is completed. The arguments for steps-$(II),(V)$ are same as that of previous case (i.e the case $C.D=2$). Step-$(III)$ follows from step-$(II)$ and the assumption $(iii)(b)$ first part. Step-$(VI)$ follows from Proposition \ref{VAN}.\\

Let $12 \leq C.D \leq 15$.  Note that, in this situation we have, $C.(2C-D)<0$, for $13 \leq C.D \leq 15$ and for $C.D=12$, we have $|2C-D| = \emptyset$ (from assumption $(iii)(c)$ first part and S.E.S \ref{e2}). This means $h^0(\mathcal O_C(D)) \geq 6$ (by relation \ref{K} and S.E.S \ref{d1}).
From Corollary \ref{wH1} and from assumption $(iii)(c)$ second part we have $h^0(\mathcal O_C(D))=h^0(\mathcal O_X(D))=6$. Then from relation \ref{K}, step-$(I)$ is completed. From step-$(I)$ one can see that step-$(II)$ follows. Step-$(III)$ follows from step-$(II)$ and assumption $(iii)(b)$ first part. Step-$(IV)$ follows from Step-$(III)$ and assumption $(iii)(b)$ second  part. Step-$(V)$ follows from step-$(I)$ and from the facts $|C-D| =\emptyset, h^0(\mathcal O_X(2C-D))=h^0(\mathcal O_C(2C-D))=0$. Finally step-$(VI)$ follows from step-$(V)$ and the observation $C.(C-D) <0$.\\

Consider \underline{\textbf{$P_a(D)=2C.D-5$.}}

Let $3 \leq C.D \leq 4$. Here we have by Remark \ref{R1}, S.E.S \ref{d1}, S.E.S \ref{TES} and initializedness, $h^0(\mathcal O_C(D))=h^0(\mathcal O_X(D))=1, h^0(\mathcal O_X(C-D)) \leq 1$. Then using the relation \ref{K} and applying long exact cohomology sequence to the S.E.S \ref{e2}, we have 
$h^0(\mathcal O_C(2C-D)) \geq 3$.  Corollary \ref{wH1} forces $h^0(\mathcal O_C(2C-D))=3, h^0(\mathcal O_X(C-D))=h^0(\mathcal O_C(C-D))=1, h^0(\mathcal O_X(2C-D))=4$. Then from the relation \ref{K} step-$(I)$ is done. Step-$(II)$ follows from step-$(I)$ and the observation mentioned in the beginning of this paragraph. It's easy to see that step-$(V)$ follows from step-$(I)$ (and the cohomological information obtained therein) and step-$(VI)$ follows from step-$(I)$ and step-$(V)$.\\
 


Let $ C.D=5 $.  By Remark \ref{R1}, we have $h^0(\mathcal O_C(D)) \leq 2$. Note that, by Lemma \ref{L2}, one has $|C-D|= \emptyset$. Then using the relation \ref{K} and S.E.S  \ref{e2}, \ref{d1}, we have $h^0(\mathcal O_C(2C-D)) \geq 3$.  Corollary \ref{wH1} forces $h^0(\mathcal O_C(2C-D))=h^0(\mathcal O_X(2C-D))=3, h^0(\mathcal O_X(D))=h^0(\mathcal O_C(D))=2$. Then from  relation \ref{K} step-$(I)$ is done. Steps-$(II),(V)$ follows from step-$(I)$ and the cohomological data obtained there.  Step-$(VI)$ follows from the Proposition \ref{VAN}.\\


Let $ C.D=6 $. Again we have by Remark \ref{R1}, $h^0(\mathcal O_C(D)) \leq 3$. Note that, we  have,  $|C-D|= \emptyset$. Then using the relation \ref{K} and S.E.S  \ref{e2}, \ref{d1}, we have $h^0(\mathcal O_C(2C-D)) \geq 2$. It can be easily deduced that $h^0(\mathcal O_X(D))\neq 1$. Then assumption $(iv)(c)$ first part, the abovementioned S.E.S's and relation \ref{K} force that we have either $h^0(\mathcal O_C(2C-D))=h^0(\mathcal O_X(2C-D))=3, h^0(\mathcal O_X(D))=h^0(\mathcal O_C(D))=2$ or $h^0(\mathcal O_C(2C-D))=h^0(\mathcal O_X(2C-D))=2, h^0(\mathcal O_X(D))=h^0(\mathcal O_C(D))=3$. Then from the relation \ref{K}, step-$(I)$ is done. Arguments for steps-$(II), (V)$ is same as that of the case $C.D=5$. Step-$(III)$ follows from step-$(II)$ and assumption $(iv)(b)$ first part. It can be seen that step-$(VI)$ follows from step-$(V)$ and assumption $(iv)(c)$ second part.\\


Let $ C.D=7 $. In this case we have by Corollary \ref{wH1}, $h^0(\mathcal O_C(D)) \leq 3$ . Then using the relation \ref{K} and S.E.S  \ref{e2}, \ref{d1}, one has $h^0(\mathcal O_C(2C-D)) \geq 2   $.   Remark \ref{R1} forces $h^0(\mathcal O_C(2C-D))=h^0(\mathcal O_X(2C-D))=2$ and hence from relation \ref{K} and S.E.S \ref{d1}, one sees that $h^0(\mathcal O_X(D))=h^0(\mathcal O_C(D))=3$. Then from  relation \ref{K} step-$(I)$ is done. Again the arguments for steps-$(II), (V)$ is same as that of the case $C.D=6$. Step-$(III)$ follows from step-$(II)$ and assumption $(iv)(b)$ first part. Finally note that, step-$(VI)$ follows from step-$(V)$ and the observation that $C.(C-D)<0$.\\



Let $ C.D=10$. In this case we have by Remark \ref{R1}, $h^0(\mathcal O_C(2C-D)) \leq 1$ . Then using the relation \ref{K} and S.E.S  \ref{e2}, we have $h^0(\mathcal O_X(D)) \geq 4   $.  Corollary \ref{wH1} and S.E.S \ref{d1} forces $h^0(\mathcal O_C(2C-D))=h^0(\mathcal O_X(2C-D))=1, h^0(\mathcal O_X(D))=h^0(\mathcal O_C(D))=4$. Then from the relation \ref{K} step-$(I)$ is done. The arguments for the remaining cases are same as the case $C.D=7$.\\


Let $ C.D=11$.  By assumption $(iv)(b)$ second part, \ref{K} and S.E.S  \ref{e2}, we have $h^0(\mathcal O_X(D)) \geq 5   $.  Corollary \ref{wH1} and S.E.S \ref{d1} forces $ h^0(\mathcal O_X(D))=h^0(\mathcal O_C(D))=5$. Then from the relation \ref{K} step-$(I)$ is done. The arguments for the remaining cases are same as the case $C.D=10$.\\


Finally let $12 \leq C.D \leq 14$. Note that, in each such situation one has $h^0(\mathcal O_C(2C-D))=0$. Then by \ref{K} and S.E.S  \ref{e2}, we have $h^0(\mathcal O_X(D)) \geq 5   $.  Assumption $(iv)(d)$ and S.E.S \ref{d1} forces $ h^0(\mathcal O_X(D))=h^0(\mathcal O_C(D))=5$. Then from the relation \ref{K} step-$(I)$ is done. Step-$(II)$ follows from step-$(I)$ and the cohomological equalities obtained there. Step-$(III)$ follows from step-$(II)$ and assumption $(iv)(b)$ first part. Step-$(IV)$ follows from step-$(III)$ and  the observation $h^0(\mathcal O_C(D))<6$. It's easy to see that step-$(V)$ follows from step-$(I)$ and assumption $(iv)(b)$ second part. Finally step-$(VI)$ follows by the same reason as $C.D=10$.\\\

Consider  \underline{\textbf{$P_a(D)=2C.D-6$.}}

Let $3 \leq C.D \leq 4$. As before by Remark \ref{R1}, S.E.S \ref{d1}, S.E.S \ref{TES} and initializedness, $h^0(\mathcal O_C(D))=h^0(\mathcal O_X(D))=1, h^0(\mathcal O_X(C-D)) \leq 1$. From the relation \ref{K} and S.E.S \ref{e2}, we have $h^0(\mathcal O_C(2C-D)) \geq 2$. Note that, by Corollary \ref{wH1}, $h^0(\mathcal O_C(2C-D)) \leq 3$. Then one can deduce using the assumption $(v)(c)$ first part, relation \ref{K} and S.E.S \ref{e2} that either   
  we have $h^0(\mathcal O_X(C-D))=h^0(\mathcal O_C(C-D)) =1$ and $h^0(\mathcal O_X(2C-D))=3, h^0(\mathcal O_C(2C-D)) = 2$ or
   $h^0(\mathcal O_X(C-D))=h^0(\mathcal O_C(C-D)) =0$ and $h^0(\mathcal O_X(2C-D))= h^0(\mathcal O_C(2C-D)) = 3$. Then from the relation \ref{K} step-$(I)$ is completed. The remaining steps follow by the same arguments as that of the case $P_a(D)=2C.D-5, 3 \leq C.D \leq 4$.\\

Let $ C.D=5 $.  By Remark \ref{R1}, we have $h^0(\mathcal O_C(D)) \leq 2$. Note that, by Lemma \ref{L2}, one has $|C-D|= \emptyset$. Then using the relation \ref{K} and S.E.S  \ref{e2}, \ref{d1}, we have $h^0(\mathcal O_C(2C-D)) \geq 2$. By  Corollary \ref{wH1}, relation \ref{K} and assumption $(v)(f)$ first part, one can have either  $h^0(\mathcal O_C(2C-D))=h^0(\mathcal O_X(2C-D))=3, h^0(\mathcal O_X(D))=h^0(\mathcal O_C(D))=1$ or $h^0(\mathcal O_C(2C-D))=h^0(\mathcal O_X(2C-D))=2, h^0(\mathcal O_X(D))=h^0(\mathcal O_C(D))=2$. Then from the relation \ref{K} step-$(I)$ is done. The remaining steps follow by the same arguments as that of the case $P_a(D)=2C.D-5, C.D=5$.\\

Let $ C.D=6 $. As before by Remark \ref{R1}, one has $h^0(\mathcal O_C(D)) \leq 3$. Note that, we  have,  $|C-D|= \emptyset$. Then using the relation \ref{K} and S.E.S  \ref{e2}, \ref{d1}, we have $h^0(\mathcal O_C(2C-D)) \geq 1$.  Assumption $(v)(e)$ (and the abovementioned S.E.S's) forces that we have either $h^0(\mathcal O_C(2C-D))=h^0(\mathcal O_X(2C-D))=3, h^0(\mathcal O_X(D))=h^0(\mathcal O_C(D))=1$ or $h^0(\mathcal O_C(2C-D))=h^0(\mathcal O_X(2C-D))=2, h^0(\mathcal O_X(D))=h^0(\mathcal O_C(D))=2$ or $h^0(\mathcal O_C(2C-D))=h^0(\mathcal O_X(2C-D))=1, h^0(\mathcal O_X(D))=h^0(\mathcal O_C(D))=3$ . Then from the relation \ref{K} step-$(I)$ is completed. Steps-$(II),(V)$ follow from step-$(I)$ and the cohomological observations obtained there. Step-$(III)$ follows from step-$(II)$ and assumption $(v)(b)$ first part. Then step-$(VI)$ follows from step-$(V)$ and assumption $(v)(c)$ second part.\\


Let $ C.D=7 $. In this case we have by Corollary \ref{wH1}, $h^0(\mathcal O_C(D)) \leq 3$. Then using the relation \ref{K} and S.E.S  \ref{e2}, \ref{d1}, we have $h^0(\mathcal O_C(2C-D)) \geq 1$.  Then Remark \ref{R1}, assumption $(v)(d)$, relation \ref{K} forces either $h^0(\mathcal O_C(2C-D))=h^0(\mathcal O_X(2C-D))=2, h^0(\mathcal O_X(D))=h^0(\mathcal O_C(D))=2$ or $h^0(\mathcal O_C(2C-D))=h^0(\mathcal O_X(2C-D))=1, h^0(\mathcal O_X(D))=h^0(\mathcal O_C(D))=3$. Then from the relation \ref{K} step-$(I)$ is done. Steps-$(II),(V)$ follow from step-$(I)$ and the observations mentioned there. Step-$(III)$ follows from step-$(II)$ and assumption $(v)(b)$ first part. Then step-$(VI)$ follows from step-$(V)$ and the observation that $C.(C-D)<0$.\\

Let $8 \leq C.D \leq 9 $. Then we have by Corollary \ref{wH1}, $h^0(\mathcal O_C(D)) \leq 3$. Then using Remark \ref{R1}, relation \ref{K} and S.E.S  \ref{e2}, one can see that $h^0(\mathcal O_X(D)) \leq 2$ isn't possible.  Then Remark \ref{R1}, relation \ref{K} yields $h^0(\mathcal O_C(2C-D))=h^0(\mathcal O_X(2C-D))=1, h^0(\mathcal O_X(D))=h^0(\mathcal O_C(D))=3$. Then from the relation \ref{K} step-$(I)$ is done. The arguments for the remaining cases are same as that of the previous case (i.e., $C.D=7$).\\

Let $ C.D=10$.  Then using the relation \ref{K} and S.E.S  \ref{e2} and the assumption $(v)(b)$ second part, we have $h^0(\mathcal O_X(D)) \geq 4   $.  Corollary \ref{wH1} and S.E.S \ref{d1} forces $ h^0(\mathcal O_X(D))=h^0(\mathcal O_C(D))=4$. Then from the relation \ref{K} step-$(I)$ is done. The arguments for the remaining cases are same as that of the  case $C.D=7$.\\

Let $ C.D=11$. In this case we have by assumption $(v)(b)$ second part, relation \ref{K} and S.E.S  \ref{e2},  $h^0(\mathcal O_X(D)) \geq 4   $.  Then assumption $(v)(e)$ second part and S.E.S \ref{d1} enable us to obtain  $ h^0(\mathcal O_X(D))=h^0(\mathcal O_C(D))=4$. Then from the relation \ref{K} step-$(I)$ is done. The arguments for the remaining cases are same as that of the  case $C.D=7$.\\

Finally let $12 \leq C.D \leq 13$. Note that, in each such situation one has $h^0(\mathcal O_C(2C-D))=0$. Then by \ref{K} and S.E.S  \ref{e2}, we have $h^0(\mathcal O_X(D)) \geq 4   $.  Assumption $(v)(f)$ second part and S.E.S \ref{d1} forces $ h^0(\mathcal O_X(D))=h^0(\mathcal O_C(D))=4$. Then from the relation \ref{K} step-$(I)$ is done. Step-$(II)$ follows from step-$(I)$ and the cohomological equality obtained there. Step-$(III)$ follows from step-$(II)$ and assumption $(v)(b)$ first part. Step-$(IV)$ follows from step-$(III)$ and  the observation $h^0(\mathcal O_C(D))<6$. It's easy to see that step-$(V)$ follows from step-$(I)$, assumption $(v)(b)$ second part (for $C.D=12$) and the observation $C.(2C-D)<0$ (for $C.D=13$). Finally step-$(VI)$ follows by the same reason as $C.D=10$.\\

Consider \underline{\textbf{$P_a(D)=2C.D-7$.}}

Let $C.D=4$. By Remark \ref{R1}, S.E.S \ref{d1} and initializedness it follows that $h^0(\mathcal O_C(D)) =h^0(\mathcal O_X(D))=1$.  Note that, one can use Lemma \ref{L3} to obtain $|C-D|= \emptyset$.  From the relation \ref{K} and S.E.S \ref{e2}, we have $h^0(\mathcal O_C(2C-D)) \geq 2$. Then one can deduce using the assumption $(vi)(g)$, relation \ref{K} and S.E.S \ref{e2}, $h^0(\mathcal O_X(2C-D))= h^0(\mathcal O_C(2C-D)) = 2$ . Then from  relation \ref{K} step-$(I)$ is completed. Step-$(II)$ follows from step-$(I)$ and the observation mentioned in the beginning of this paragraph. It's clear that step-$(V)$ follows from step-$(I)$ and step-$(VI)$ follows from assumption $(vi)(h)$ first part and step-$(V)$.\\

Let $ C.D=5 $. In this case we have by Remark \ref{R1}, $h^0(\mathcal O_C(D)) \leq 2$. Note that, by Lemma \ref{L2}, one has $|C-D|= \emptyset$. Then using the relation \ref{K} and S.E.S  \ref{e2}, we have $h^0(\mathcal O_C(2C-D)) \geq 1$. By  corollary \ref{wH1}, relation \ref{K} and assumption $(vi)(f)$  one can have either  $h^0(\mathcal O_C(2C-D))=h^0(\mathcal O_X(2C-D))=2, h^0(\mathcal O_X(D))=h^0(\mathcal O_C(D))=1$ or $h^0(\mathcal O_C(2C-D))=h^0(\mathcal O_X(2C-D))=1, h^0(\mathcal O_X(D))=h^0(\mathcal O_C(D))=2$. Then from the relation \ref{K} step-$(I)$ is done.  Steps-$(II),(V)$ follow by the same arguments as that of the case $P_a(D)=2C.D-5, C.D=5$. Step-$(VI)$ follows from step-$(V)$ and Proposition \ref{VAN}.\\

Let $ C.D=6 $. Then by Remark \ref{R1}, one has $h^0(\mathcal O_C(D)) \leq 3$. Note that, we  have,  $|C-D|= \emptyset$. Then using  assumption $(vi)(e)$, relation \ref{K}, S.E.S \ref{e2}, S.E.S \ref{d1},   we have either $h^0(\mathcal O_C(2C-D))=h^0(\mathcal O_X(2C-D))=0, h^0(\mathcal O_X(D))=h^0(\mathcal O_C(D))=3$ or $h^0(\mathcal O_C(2C-D))=h^0(\mathcal O_X(2C-D))=2, h^0(\mathcal O_X(D))=h^0(\mathcal O_C(D))=1$ or $h^0(\mathcal O_C(2C-D))=h^0(\mathcal O_X(2C-D))=1, h^0(\mathcal O_X(D))=h^0(\mathcal O_C(D))=2$ . Therefore, from the relation \ref{K} step-$(I)$ is done. Steps-$(II), (V)$ follow from step-$(I)$ and the observation obtained there. Step-$(III)$ follows from step-$(II)$ and assumption $(vi)(b)$ first part. Then step-$(VI)$ follows from step-$(V)$ and assumption $(vi)(h)$ first part.\\

Let $ C.D=7 $. In this case we have by Remark \ref{R1}, $h^0(\mathcal O_C(2C-D)) \leq 2$. Then using the relation \ref{K} and S.E.S  \ref{e2}, we have $h^0(\mathcal O_X(D)) \geq 1$.  Then Remark \ref{R1}, Corollary \ref{wH1}, assumption $(vi)(d)$, relation \ref{K},  forces either $h^0(\mathcal O_C(2C-D))=h^0(\mathcal O_X(2C-D))=2, h^0(\mathcal O_X(D))=h^0(\mathcal O_C(D))=1$ or $h^0(\mathcal O_C(2C-D))=h^0(\mathcal O_X(2C-D))=1, h^0(\mathcal O_X(D))=h^0(\mathcal O_C(D))=2$ or $h^0(\mathcal O_C(2C-D))=h^0(\mathcal O_X(2C-D))=0, h^0(\mathcal O_X(D))=h^0(\mathcal O_C(D))=3$. Thus from  relation \ref{K} step-$(I)$ is done. Steps-$(II), (V)$ follows by the same arguments as in the prvious case (i.e $C.D=6$). Step-$(III)$ follows from step-$(II)$ and assumption $(vi)(b)$ first part. Then step-$(VI)$ follows from step-$(V)$ and the observation that $C.(C-D)<0$.\\

Let $8 \leq C.D \leq 9 $. Then we have by Corollary \ref{wH1}, $h^0(\mathcal O_C(D)) \leq 3$. From Remark \ref{R1}, relation \ref{K} and S.E.S  \ref{e2}, one can see that $h^0(\mathcal O_X(D))= 1$ isn't possible.  Then Remark \ref{R1}, relation \ref{K}, assumption $(vi)(c)$ forces either $h^0(\mathcal O_C(2C-D))=h^0(\mathcal O_X(2C-D))=1, h^0(\mathcal O_X(D))=h^0(\mathcal O_C(D))=2$ or $h^0(\mathcal O_C(2C-D))=h^0(\mathcal O_X(2C-D))=0, h^0(\mathcal O_X(D))=h^0(\mathcal O_C(D))=3$. Therefore, from  relation \ref{K}, step-$(I)$ is done. The arguments for the remaining cases are same as that of the previous case(i.e $C.D=7$).\\


Let $ 10 \leq  C.D \leq 11$. In this case we have by assumption $(vi)(b)$ second part, relation \ref{K} and S.E.S  \ref{e2}, we have $h^0(\mathcal O_X(D)) \geq 3  $.  From assumption $(vi)(h)$ second part and S.E.S \ref{d1}, we have $ h^0(\mathcal O_X(D))=h^0(\mathcal O_C(D))=3$, whence step-$(I)$ is acheived from relation \ref{K}. The arguments for the remaining cases are same as that of the  case $C.D=7$.\\

Finally let $C.D=12$. In this case all the steps except step-$(IV)$ follow by the same argument as that of case $C.D=10$. It is easy to see that step-$(IV)$ follows from step-$(III)$ and  the observation $h^0(\mathcal O_C(D))<6$.\\

Consider \underline{\textbf{$P_a(D)=2C.D-8$.}}

Let $C.D=4$. As before by Remark \ref{R1}, S.E.S \ref{d1} and initializedness, one has $h^0(\mathcal O_C(D)) =h^0(\mathcal O_X(D))=1$  and by Lemma \ref{L3} one obtains $|C-D|= \emptyset$.  Then  relation \ref{K}, S.E.S \ref{e2} and the assumption $(vii)(d)$ first part enables us to deduce $h^0(\mathcal O_X(2C-D)) =h^0(\mathcal O_C(2C-D))=1$, whence  step-$(I)$ is completed by relation \ref{K}. Again as before, step-$(II)$ follows from initializedness, step-$(I)$ and the observation mentioned at the beginning of this paragraph. It is easy to see that step-$(V)$ follows from the cohomological information obtained in step-$(I)$ and step-$(VI)$ follows from the assumption $(vii)(c)$ first part  and step-$(V)$.\\

Let $ C.D=5 $. Then by Remark \ref{R1}, we have $h^0(\mathcal O_C(D)) \leq 2$. Again as in previous situations, by Lemma \ref{L2}, one has $|C-D|= \emptyset$. Then using  relation \ref{K}, S.E.S  \ref{e2}, S.E.S \ref{d1}, initializedness, assumptions $(vii)(b)$ first part and $(vii)(c)$ second part, we have either  $h^0(\mathcal O_C(2C-D))=h^0(\mathcal O_X(2C-D))=1, h^0(\mathcal O_X(D))=h^0(\mathcal O_C(D))=1$ or $h^0(\mathcal O_C(2C-D))=h^0(\mathcal O_X(2C-D))=0, h^0(\mathcal O_X(D))=h^0(\mathcal O_C(D))=2$. Then from  relation \ref{K} step-$(I)$ is done.  Steps-$(II),(V)$ follow by the same arguments as that of the case $P_a(D)=2C.D-5, C.D=5$. Step-$(VI)$ follows from step-$(V)$ and the assumption $(vii)(c)$ first part.\\

Let $ C.D=6 $.  By Remark \ref{R1} and assumption $(vii)(b)$ second part, one has $h^0(\mathcal O_C(D)) \leq 2$. Note that,  by initializedness, we have $|C-D|= \emptyset$. Then   relation \ref{K}, S.E.S \ref{e2}, S.E.S \ref{d1}, initializedness and assumption $(vii)(c)$ second part forces  either $h^0(\mathcal O_C(2C-D))=h^0(\mathcal O_X(2C-D))=0, h^0(\mathcal O_X(D))=h^0(\mathcal O_C(D))=2$ or $h^0(\mathcal O_C(2C-D))=h^0(\mathcal O_X(2C-D))=1, h^0(\mathcal O_X(D))=h^0(\mathcal O_C(D))=1$. Then from the relation \ref{K} step-$(I)$ is done. Step-$(II)$ follows from initializedness, step-$(I)$ and the cohomological observations obtained therein. Step-$(III)$ follows from step-$(II)$ and statement asserted in the second line of this paragraph. Again as in other cases step-$(V)$ follows from the cohomological values obtained in step-$(I)$ and step-$(VI)$ follows from step-$(V)$ and assumption $(vii)(c)$ first part.\\

Let $ C.D=7 $. In this situation we have by Corollary \ref{wH1} and assumption $(vii)(b)$ second part, we have $h^0(\mathcal O_C(D)) \leq 2$. Then  relation \ref{K} and S.E.S  \ref{e2},  S.E.S \ref{d1}, initializedness, Remark \ref{R1} and assumptions $(vii)(b)$ first part with $(vii)(c)$ second part forces either $h^0(\mathcal O_C(2C-D))=h^0(\mathcal O_X(2C-D))=1, h^0(\mathcal O_X(D))=h^0(\mathcal O_C(D))=1$ or $h^0(\mathcal O_C(2C-D))=h^0(\mathcal O_X(2C-D))=0, h^0(\mathcal O_X(D))=h^0(\mathcal O_C(D))=2$, whence step-$(I)$ is evident from relation \ref{K}. Step-$(II)$ follows from step-$(I)$ and the cohomological observations obtained in the previous line. Step-$(III)$ follows from step-$(II)$ and the upperbound mentioned in the second line of this paragraph. It's clear that Step-$(V)$ follows from step-$(I)$ and two sets of cohomological possibilities seen therein. Then step-$(VI)$ follows from step-$(V)$ and the fact that $C.(C-D)<0$.\\

Let $8 \leq C.D \leq 9 $. Note that, by Remark \ref{R1}, one has $h^0(\mathcal O_C(2C-D))\leq 1$. Then the rest of the arguments are the same as in the previous case (i.e., $C.D=7$). \\


Let $ 10 \leq  C.D \leq 11$. By $(vii)(d)$ second part, S.E.S \ref{d1}, initializedness, relation \ref{K} and S.E.S  \ref{e2}, one obtains $h^0(\mathcal O_C(2C-D))=h^0(\mathcal O_X(2C-D))=0, h^0(\mathcal O_X(D))=h^0(\mathcal O_C(D))=2$. Therefore, step-$(I)$ is completed by relation \ref{K}. Arguments for step-$(II),(V),(VI)$ are same as the case $C.D=7$. Finally step-$(III)$ follows from assumption $(vii)(d)$ second part.\\

Consider \underline{\textbf{$P_a(D)=2C.D-9 $.}}\

Note that, in each situation by assumption $(viii)(a)$, S.E.S \ref{d1}, initializedness, relation \ref{K} and S.E.S  \ref{e2}, one must have $h^0(\mathcal O_C(2C-D))=h^0(\mathcal O_X(2C-D))=0, h^0(\mathcal O_X(D))=h^0(\mathcal O_C(D))=1$. Then from relation \ref{K}, step-$(I)$ follows. Arguments for step-$(II),(V)$ are the same as before (i.e., case $P_a(D)=2C.D-8$, $C.D=7$).   For $5 \leq C.D \leq 6$, step-$(VI)$ follows from assumption $(viii)(b)$ and step-$(V)$. As earlier for $7 \leq C.D \leq 10$, step-$(VI)$ follows from step-$(V)$ and the observation that $C.(C-D)<0$. Lastly, for $6 \leq C.D \leq 10$, step-$(III)$ follows from step-$(II)$ and assumption $(viii)(a)$.

\end{proof}

We conclude this paper by pointing out a concrete example where we have $D$ is a non-zero effective divisor on some smooth sextic surface in $\mathbb P^3$ such that $P_a(D)=2C.D-9$, $C.D=5$ and $h^0(\mathcal O_X(D))=1$ but $h^0(\mathcal O_X(2C-D)) \neq 0$ (i.e., $\mathcal O_X(D)$ is not an ACM line bundle). This also gives us a naive understanding of why we are forced to impose the condition $h^0(\mathcal O_C(2C-D))=0$ in case $(viii)(a)$.\

\begin{example}(See \cite{Counterexample})
Let $A$ be a plane conic and $B$ be a plane cubic curve  in $\mathbb P^3$(lying in different planes) which meets transversally at $1$ point. Then it can be shown that there exists a smooth sextic surface in $\mathbb P^3$ containing $A$ and $B$ (say $X$). Denote $D:=A+B$. If $C$ is a hyperplane section of $X$, then we have $C.D=5$. Note that, $\text{deg}(K_A)= K_X.A+A^2$ gives us  $A^2=-6$. Similarly, one obtains $B^2=-6$. This in turn, by construction, means $D^2=-10$ and hence $P_a(D) =2C.D-9$ for $C.D=5$. The existence of non-trivial sections of $\mathcal O_X(2C-D)$ follows from construction.
\end{example}

\section{Acknowledgement}
 
 I thank Dr. Sarbeswar Pal for many  discussions and anonymous mathoverflow user abx for suggesting the article \cite{KC}.





\end{document}